\documentclass[a4, 11pt] {amsart} 
\usepackage{amssymb,times, amscd,amsmath,amsthm, xypic}
\usepackage{geometry}
\usepackage{amsmath}
\usepackage{amssymb}
\usepackage[all]{xy}
\usepackage[dvips]{graphicx, color}
\usepackage{mathrsfs}

\numberwithin{equation}{section}
\newtheorem{thm}[equation]{Theorem}
\newtheorem{pro}[equation]{Proposition}

\newtheorem{cor}[equation]{Corollary}

\newtheorem{lem}[equation]{Lemma}
\theoremstyle{definition}
\newtheorem{ex}[equation]{Example}
\newtheorem{defn}[equation]{Definition}
\newtheorem{ob}[equation]{Observation}
\newtheorem{rem}[equation]{Remark}

\def\alp{\alpha}

\def\la{\lambda}

\def\op{\operatorname}

\begin{document}

\title{Decomposition spaces and poset-stratified spaces} 

\author{
Shoji Yokura}

\date{}

\address{Department of Mathematics  and Computer Science, Graduate School of Science and Engineering, Kagoshima University, 1-21-35 Korimoto, Kagoshima, 890-0065, Japan
}

\email{yokura@sci.kagoshima-u.ac.jp}

\maketitle 

\begin{abstract} In 1920s R. L. Moore introduced \emph{upper semicontinuous} and \emph{lower semicontinuous} decompositions in studying decomposition spaces. Upper semicontinuous decompositions were studied very well by himself and later by R.H. Bing in 1950s. In this paper we consider lower semicontinuous decompositions $\mathcal D$ of a topological space $X$ such that the decomposition spaces $X/\mathcal D$ are Alexandroff spaces. If the associated proset (preordered set) of the decomposition space $X/\mathcal D$ is a poset, then the decomposition map $\pi:X \to X/\mathcal D$ is \emph{a continuous map from the topological space $X$ to the poset $X/\mathcal D$ with the associated Alexandroff topology}, which is nowadays
called \emph{a poset-stratified space}. As an application, we capture the face poset of a real hyperplane arrangement $\mathcal A$ of $\mathbb R^n$ as the associated poset of the decomposition space $\mathbb R^n/\mathcal D(\mathcal A)$ of the decomposition $\mathcal D(\mathcal A)$ determined by the arrangement $\mathcal A$. We also show that for any locally small category $\mathcal C$ the set $hom_{\mathcal C}(X,Y)$ of morphisms from $X$ to $Y$ can be considered as a poset-stratified space, and that for any objects $S, T$ (where $S$ plays as a source object and $T$ as a target object) there are a covariant functor $\frak {st}^S_*: \mathcal C \to \mathcal Strat$ and 
a contravariant functor $\frak {st}^*_T$ $\frak {st}^*_T: \mathcal C \to \mathcal Strat$ from $\mathcal C$ to the category $\mathcal Strat$ of poset-stratified spaces. We also make a remark about Yoneda's Lemmas as to poset-stratified space structures of $hom_{\mathcal C}(X,Y)$.
\end{abstract}

\section{Introduction}

 Let $X$ be a topological space and $\mathcal D = \{D_{\lambda} \subset X\}_{\lambda\in \Lambda}$ be a decomposition of $X$, i.e., $X = \bigcup_{\lambda \in \Lambda} D_{\lambda}$ such that $D_{\lambda} \cap D_{\mu} = \emptyset$ for $\lambda \not = \mu.$
Then, by considering each subset $D_{\lambda}$ as a point, we get the quotient map $\pi: X \to X/\mathcal D$. If we impose the quotient topology $\tau_{\pi}$ on the target $X/\mathcal D$, i.e., the finest or strongest topology such that the quotient map $\pi: X \to X/\mathcal D$ becomes continuous, then the topological space $(X/\mathcal D, \tau_{\pi})$ is called \emph{the decomposition space} and the continuous map $\pi: X \to X/\mathcal D$ is called \emph{the decomposition map}.

As to decompositions of space, R. L. Moore (\cite{Moore}, cf. \cite{Moore2}) introduced the notions of \emph{upper semicontinuous decomposition} and \emph{lower semicontinuous decomposition}, which turn out to correspond to the decomposition map $\pi: X \to X/\mathcal D$ being \emph{a closed map} and \emph{an open map} respectively.

In this paper 
we consider 
lower semicontinuous decompositions, i.e., decompositions such that the decomposition maps are open. To see if the decomposition map is open or not, we appeal to the proset (preordered set) structure of a topological space. Given a topological space $(X, \tau)$, we define the associated preorder $x \leqq_{\tau} y$ by $x \in \overline {\{y\}}$, which gives us a proset $(X, \leqq_{\tau})$. Conversely, given a proset $(X, \leqq)$ we define the associated topology (i.e., open sets) $\tau_{\leqq}$ by $U \in \tau_{\leqq} \Longleftrightarrow x \in U, x \leqq y \Rightarrow y \in U$. We always have $\leqq_{\tau_{\leqq}} = \leqq$, but in general $\tau \not = \tau_{\leqq_{\tau}}$. However, if a topological space is an Alexandroff space, i.e., the intersection of any family of open sets is open or equivalently the union of any family of closed sets is closed (e.g., see \cite{A1, Ar, May1, Sp}; for example any finite topological space and the associated topological space $(X, \tau_{\leqq})$ of any proset $(X, \leqq)$ are Alexandroff spaces), then we do have $\tau = \tau_{\leqq_{\tau}}$. Hence Alexandorff spaces and prosets are equivalent. So, when the decomposition space is an Alexandroff space, using Tamaki's result \cite{Tamaki} we can show that if $\mathcal D = \{D_{\lambda}\}_{\lambda\in \Lambda}$ is a decomposition of $X$ such that the decomposition space is Alexandroff, then the decomposition map $\pi: X \to (X/\mathcal D, \tau_{\pi})$ is open if and only if $\lambda \leqq_{\tau_{\pi}} \mu \Leftrightarrow D_{\lambda} \subset \overline{D_{\mu}}$.

If the associated proset $(X/\mathcal D, \leqq_{\tau_{\pi}})$ is a poset (partially ordered set), then each piece $D_{\lambda}$ has to be a locally closed set, i.e., the intersection of an open set and a closed set, because for any poset a singleton is a locally close set in the associated topological space (Alexandroff space). If the decomposition map $\pi: X \to (X/\mathcal D, \tau_{\pi})$ is open, then the associated proset $(X/\mathcal D, \leqq_{\tau_{\pi}})$ is a poset if and only if each piece $D_{\lambda}$ is a locally closed set.

A stratification is a well-known notion in geometry and topology and its definition depends on objects to study, such as topological stratification and Thom--Whitney stratification etc. (see \cite{Tamaki} for a nice review of several stratifications). We consider the following seemingly general one:
If a family $\{D_{\la}\}_{\la \in \Lambda}$ of subsets of a topological space $X$ satisfies the following conditions, then $\{D_{\la}\}_{\la \in \Lambda}$  is called \emph{a stratification} of $X$.
\begin{enumerate}
\item $D_{\la} \cap D_{\mu} = \emptyset$ if $\la \not = \mu.$
\item $X = \bigcup_{\la} D_{\la}$.
\item (locally closed set) Each $D_{\la}$ is a locally closed set
\item (frontier condition) $D_{\la} \cap \overline{D_{\mu}} \not = \emptyset  \Longrightarrow D_{\la} \subset \overline{D_{\mu}}.$
\end{enumerate}

In particular, for a finite stratification (i.e., $|\Lambda| < \infty$, which is the case in most cases in algebraic geometry and topology), using Hiro Lee Tanaka's result \cite{Tam2} (see Theorem \ref{Tanaka} in \S 4) we obtain that if $\{D_{\la}\}_{\la \in \Lambda}$ is a finite stratification, then the decomposition map $\pi: X \to (X/\mathcal D, \tau_{\pi})$ is open, thus the proset $(X/\mathcal D, \leqq_{\tau_{\pi}})$ is a poset.

A simple example of a finite stratification is $\mathcal D = \{(-\infty, 0), \{0\}, (0, \infty)\}$ of the real line $\mathbb R$. For the quotient map $\pi:\mathbb R \to \mathbb R/\mathcal D$, let $N=\pi((-\infty, 0)), O=\pi(\{0\}),P=\pi((0, \infty))$ (where $N$ stands for ``negative", $O$ ``origin", $P$ ``positive"). Then the quotient topology is $\tau_{\pi} = \{\emptyset, \{N, O, P\}, \{N\}, \{P\}, \{N, P\} \}$ and its associated poset 
$(\mathbb R/\mathcal D, \leqq_{\tau_{\pi}})$ is 
$\xymatrix{
N & O \ar[l] \ar[r] & P
},
$
where we denote $a\leqq b$ by $a \rightarrow b$ using arrow and we do not write anything for the reflexivity.

We also show that for a real hyperplane arrangement $\mathcal A$ of $\mathbb R^n$ the face poset $F(\mathcal A)$ can be captured as the associated poset $(\mathbb R^n/\mathcal D(\mathcal A), \leqq_{\tau_{\pi}})$ of the decomposition space $(\mathbb R^n/\mathcal D(\mathcal A), \tau_{\pi})$ where $\mathcal D(\mathcal A)$ is the decomposition of $\mathbb R^n$ determined by the hyperplane arrangement $\mathcal A$. The above $\mathcal D = \{(-\infty, 0), \{0\}, (0, \infty)\}$ of the real line $\mathbb R$ is nothing but the decomposition of $\mathbb R$ determined by the hyperplane arrangement $\mathcal A = \{ \{0\}\}$ of $\mathbb R$.

Now, \emph{a continuous map from a topological space to a poset with the associated Alexnadroff topology} is called \emph{a poset-stratified space} 
\cite[A.5] {Lurie} (cf. \cite[Remark 2.1.4]{AFT}). To be more precise, a pair $(X, X \xrightarrow {s} P)$ of a topological space $X$ and a continuous map $s$ from $X$ to a poset $P$ is a poset-stratified space and the continuous map $s:X \to P$ should be considered as a structure of poset-stratified space on the topological space $X$. But, unless some confusion is possible, the continuous map $s:X \to P$ is simply called a stratified space.

Classification theorems for Hurewicz fibrations have been obtained by J. Stasheff \cite{Sta}, A. Dold and R. Lashof \cite{DL} and G.Alluad \cite{All}. In our previous work \cite{YY} we study classifications of Hurewicz fibrations by considering proset structures of the homotopy set $[X,Y]$ and also on certain quotient sets of $[X,Y]$ and we get monotone maps of a proset to a poset. If we consider Alexandroff topologies (e.g., see \cite{A1}, \cite{Ar} and \cite{Sp}) of them, this map gives rise to a poset-stratified space. 
In this paper we show that in a similar way
for any locally small category $\mathcal C$ the set $hom_{\mathcal C}(X,Y)$ of morphisms between any two objects $X, Y$ can be considered as a poset-stratified space. For this, for example, we first define a preorder $\leqq_R$ for $f, g \in hom_{\mathcal C}(X,Y)$ by
$g \leqq_R f \, \Leftrightarrow \, \exists s \in hom_{\mathcal C}(X, X) \,\, \text{such that} \, f = g \circ s$, i.e., the following diagram commutes:
$$\xymatrix
{
X \ar [d]_s \ar [r]^{f} & Y\\
X \ar [ur] _{g} }
$$

Furthermore we define the equivalence relation $\sim_R$ by $f \sim_R g \, \Leftrightarrow g \leqq_R f, \, f\leqq_R g$, which means that $\exists s_1, s_2 \in hom_{\mathcal C}(X, X)$ such that $f = g \circ s_1$ and $g = f \circ s_2$, i.e., the following diagram commutes:
$$\xymatrix
{
X \ar@<.5ex> [d]^{s_1} \ar [r]^{f} & Y\\
X \ar@<.5ex> [u]^{s_2} \ar [ur] _{g} }
$$
 If we consider this for the homotopy category $h\mathcal Top$ of topological spaces, we get the following definitions: $[f], [g] \in hom_{h\mathcal Top}(X,Y) =[X,Y]$, $[g] \leqq_R[f] \, \Leftrightarrow 
[s] \in [X,X]\,\, \text{such that} \, [f] = [g] \circ [s]$, i.e., $f \sim g \circ s$, i.e., the above diagram commutes up to homotopy\footnote{As remarked in the final section, this kind of relation was already considered in a different context or for a completely different problem by Karol Borsuk and Peter Hilton (pointed out by Jim Stasheff in a private communication). If $f \sim g \circ s$, then $f$ is called \emph{a multiple of $g$} or $g$ is called \emph{ a divisor of $f$}. If $f \sim g \circ s_1$ and $g \sim f \circ s_2$, then $f$ and $g$ are called \emph{conjugate}. }.

In the final section we give an observation about the well-known Yoneda's Lemmas about representable functors.
Let $\mathcal C$ be a locally small category and $\mathcal Set$ be the category of sets. Let
$h^A(-)=hom_{\mathcal C}(-, A) : \mathcal C \to \mathcal Set$ be a representable contravariant functor and $F:\mathcal C^{op} \to \mathcal Set$ be another contravariant functor. Then we have the following canonical natural transformation (sort of ``collecting" or ``using" all the natural transformations from $h^A$ to $F$)
$$\mathcal Im_{F}:h^A(-) \to \mathscr PF(-),$$
where  for each object $X \in Obj(\mathcal C)$ we have $\mathcal Im_{F}(f) =f^*(F(A)) (\subset F(X) )$, which is the set consisting of the images of $f$ by \underline{all} the natural transformations from $h^A$ to $F$. Here $\mathscr PF(X) := \mathscr P(F(X))$ is a power set of the set $F(X)$, i.e., the set of all the subsets of $F(X)$. For any object $X$ 
$$\mathcal Im_{F}:(hom_{\mathcal C}(X,A), \leqq_L) \to (\mathscr PF(-), \leqq)$$
is a monotone map from a proset to a poset, i.e., a poset-stratified space. In other words, Yoneda's Lemma implies that any 
contravaiant (covariant, resp.) functor $F:\mathcal C^{op} \to \mathcal Set$ ($F:\mathcal C^{op} \to \mathcal Set$, resp.) induces a poset-stratified space structure on any representable contravariant (covariant, resp.) functor $hom_{\mathcal C}(-,A)$ ($hom_{\mathcal C}(A,-)$, resp.).

\section{Prosets and Alexandroff spaces}

In this section we recall some known facts of prosets, posets and Alexandroff spaces for later use.

A preorder on a set $P$ is a relation $\leqq$ which is reflexive ($a \leqq a$) and transitive ($a \leqq b, b \leqq c \Rightarrow a \leqq c$).
A set $(P, \leqq)$ equipped with a preorder $\leqq$ is called a \emph{proset} (preordered set).
If a preorder 
is anti-symmetric 
($a \leqq b, b \leqq a \Rightarrow a =b$), then it is called a partial order and a set with a partial order is called a \emph{poset} (partial ordered set).
$a \leqq b$ is also denoted by $a \to b$ using arrow.

\begin{defn}[Alexandroff topology]  Let $X$ be a topological space. If \emph{the intersection of any family} of open sets is open (or equivalently, \emph{the union of any family of closed sets is closed}), then the topology is called an \emph{Alexandroff topology} and the space is called an \emph{Alexandroff space}.
\end{defn}

For Alexandroff topology or spaces, e.g., see \cite{A1}, \cite{A2},\cite[\S 4.2.1 Alexandroff Topology]{Curry}, 
\cite[Appendix A  Pre-orders and spaces]{Woolf}.

\begin{rem} Such a space is originally called \emph{an Alexandroff-discrete space} (because he named it ``discrete R\"aume" \cite{A1}) or \emph{finitely generated space}.
We also note that any finite topological space, i.e., a finite set with a topology, is an Alexandroff space.
\end{rem}

Given a proset $(X, \leqq)$, we define
\emph{$U \subset X$ to be an open set by $x \in U, x \leqq y \Rightarrow y \in U$.}
In other words, if we let $U_x := \{y \in X \, | \, x \leqq y \}$, then $\{U_x \, | \, x \in X\}$ is the base for the topology. This topology is denoted by $\tau_{\leqq}$.

\begin{rem}
The Alexandroff topology is sometimes considered by defining an open set to be \emph{closed downwards} instead of closed upwards, e.g., see \cite{Ar}, \cite{B}, \cite{May1} and \cite{Sp}. When stratification theory or poset-stratified spaces are considered as in the above cited references \cite{Curry} and \cite{Woolf}, upward closeness is used in defining Alexandroff topology (e.g., see \cite[Definition A.5.1]{Lurie} and \cite[Definition 2.1 ]{Tamaki} as well).
\end{rem}
\begin{lem} For a proset $(X, \leqq)$, the topological space $(X, \tau_{\leqq})$ is an Alexandroff space.
\end{lem}

Because of this, the topology $\tau_{\leqq}$ is called the Alexandroff topology (associated to the preorder).

\begin{lem}\label{monotone} If $f:(X, \leqq_1) \to (Y, \leqq_2)$ is a monotone function, i.e., $x \leqq_1 y$ implies $f(x) \leqq_2 f(y)$, then
$f:(X, \tau_{\leqq_1}) \to (Y, \tau_{\leqq_2})$ is a continuous map.
\end{lem}

Let $\mathcal Proset$ be the category of prosets and monotone functions of prosets and $\mathcal Alex$ the category of Alexandroff spaces and continuous maps. Then we have a covariant functor
$\mathcal T: \mathcal Proset \to \mathcal Alex.$

Conversely, for a topological space $(X, \tau)$, we define the order $x \leqq_{\tau} y \Leftrightarrow x \in \overline {\{y\}}$, which is called \emph{specialization order}.
Certainly this is a preorder, but not necessarily a partial order.

\begin{lem} If $f:(X, \tau_1) \to (Y, \tau_2)$ is a continuous map, then $f:(X, \leqq_{\tau_1}) \to (Y, \leqq_{\tau_2}) $ is a monotone function.
\end{lem}

Therefore we have a covariant functor
$\mathcal P: \mathcal Top \to \mathcal Proset.$
For any proset $(X, \leqq)$ we have 
$$\left( \mathcal P \circ \mathcal T \right) ((X, \leqq )) = (X, \leqq ), \, \, \text{i.e., } \, \, \mathcal P \circ \mathcal T = Id_{\mathcal Proset}$$
However, in general, for a topological space $(X, \tau)$ we have
$$\left( \mathcal T \circ \mathcal P \right) ((X, \tau)) \not = (X, \tau), \, \, \text{i.e., } \, \, \mathcal T \circ \mathcal P \not = Id_{\mathcal Top}$$
The reason is simple: $\left( \mathcal T \circ \mathcal P \right) ((X, \tau))$ is always an Alexandroff space, even if the original space $(X, \tau)$ is not an Alexandroff space, namely the topology of $\left( \mathcal T \circ \mathcal P \right) ((X, \tau))$ is stronger that the original topology $\tau$. For example, let $(X, \tau)$ be a non-discrete Hausdorff space, e.g., the Euclidean space $\mathbb R^n$. Since any point of a Hausdorff space is a closed set, the order $x \leqq_{\tau} y \Leftrightarrow x \in \overline {\{y \}} =\{y\}$ implies that for the associated proset $\mathcal P((X, \tau)) = (X, \leqq_{\tau})$ we have only the following order $x \leqq_{\tau} x$ for each point $x \in X$ and there is no order for any two different points. 
Therefore 
the Alexandroff space $(\mathcal T \circ \mathcal P)((X, \tau))= (X, \tau_{\leqq_{\tau}})$ is a discrete space. 
However, if we restrict the 
functor $\mathcal P: \mathcal Top \to \mathcal Proset$ to the subcategory $\mathcal Alex$ of Alexandroff spaces, then we have 
$\left( \mathcal T \circ \mathcal P \right) ((X, \tau)) = (X, \tau), \, \, \text{i.e., } \, \, \mathcal T \circ \mathcal P = Id_{\mathcal Alex.}$
Therefore we have 
$$\mathcal P \circ \mathcal T = Id_{\mathcal Proset}, \quad \mathcal T \circ \mathcal P = Id_{\mathcal Alex}.$$
Thus Alexandroff spaces and prosets are equivalent.

\begin{pro} If we define an open set in a proset by ``up-set", then $F$ is a closed set if and only if $F$ is a ``down-set", i.e., $x \in F, y \leqq x \Rightarrow y \in F$.
\end{pro}
We get the following corollary, which will be used in later sections.
\begin{cor}\label{locally closed} Let $\Lambda$ be a poset. In the Alexandroff topological space $\mathcal T(\Lambda)$, any singleton $\{\lambda \}$ is locally closed.
\end{cor}
\begin{rem} 
If $\Lambda$ is a proset, not a poset, then it is not necessarily true that $\{\lambda\} = \{\mu | \lambda \leqq \mu \} \cap \{\mu | \mu \leqq \lambda \}.$
\end{rem}
\begin{pro}\label{product} Let $(P_i, \leqq_i)$ be a proset ($1 \leqq i \leqq n$). Then the preorder $\leqq$ of the proset $\mathcal P \left ((P_1, \tau_{\leqq_1}) \times \cdots \times (P_n, \tau_{\leqq_n}) \right)$ of the product of the Alexandroff spaces $\mathcal T((P_i, \leqq_i))= (P_i, \tau_{\leqq_i})$ is given by
$(x_1, \cdots, x_n) \leqq (y_1, \cdots, y_n) \Leftrightarrow x_1 \leqq_1 y_1, \cdots, x_n \leqq_n y_n.$
\end{pro}
\begin{proof} The product of the Alexandroff spaces is Alexandroff and the preorder $\leqq$ of an Alexandroff space is defined by
$(x_1, \cdots, x_n) \leqq (y_1, \cdots, y_n)  \Leftrightarrow (x_1, \cdots, x_n) \in \overline {\{(y_1, \cdots, y_n) \}}.$
We have $\{(y_1, \cdots, y_n) \} = \{y_1\} \times \cdots \times \{y_n\}$. Thus
$$\overline {\{(y_1, \cdots, y_n) \}} = \overline{\{y_1\} \times \cdots \times \{y_n\}} = \overline{\{y_1\}} \times \cdots \times \overline{\{y_n\}}.$$
Here the closure $\overline{\{y_i\}}$ is the closure of $\{y_i\}$ in the Alexandroff space $(P_i, \tau_{\leqq_i})$. Thus we get
$ (x_1, \cdots, x_n) \in \overline{\{y_1\}} \times \cdots \times \overline{\{y_n\}}.$
Hence we have that $x_1 \in \overline{\{y_1\}}, \cdots, x_n \in \overline{\{y_n\}}$, therefore $x_1 \leqq_1 y_1, \cdots, x_n \leqq_n y_n.$
\end{proof}
\begin{rem}
In general one can define several other preorders for the product of prosets. Following Proposition \ref{product} above, we define the preorder $\leqq_1 \times \cdots \times \leqq_n$ for the product $P_1 \times \cdots \times P_n$ by
$(x_1, \cdots, x_n) (\leqq_1 \times \cdots \times \leqq_n) (y_1, \cdots, y_n)  \Longleftrightarrow x_1 \leqq_1 y_1, \cdots, x_n \leqq_n y_n$. 
The proset $(P_1 \times \cdots \times P_n, \leqq_1 \times \cdots \times \leqq_n)$ is called \emph{the product of the prosets $(P_i, \leqq_i)$}:
$$(P_1, \leqq_1) \times \cdots \times (P_n,\leqq_n) := (P_1 \times \cdots \times P_n, \leqq_1 \times \cdots \times \leqq_n).$$
Proposition \ref{product} implies that
\begin{equation}\label{form.1}
\mathcal T ((P_1, \leqq_1) ) \times \cdots \times \mathcal T ((P_n,\leqq_n)) = \mathcal T \Bigl((P_1, \leqq_1) \times \cdots \times (P_n,\leqq_n) \Bigr ) 
\end{equation}
Indeed, 
$\mathcal P \left( \mathcal T ((P_1, \leqq_1) ) \times \cdots \times \mathcal T ((P_n,\leqq_n)) \right) = (P_1, \leqq_1) \times \cdots \times (P_n,\leqq_n)$ by Proposition \ref{product}.
Hence
$\mathcal T \left (\mathcal P \left( \mathcal T ((P_1, \leqq_1) ) \times \cdots \times \mathcal T ((P_n,\leqq_n)) \right) \right )= \mathcal T \left ((P_1, \leqq_1) \times \cdots \times (P_n,\leqq_n) \right),$
namely we have
$\left (\mathcal T \circ \mathcal P \right) \left( \mathcal T ((P_1, \leqq_1) ) \times \cdots \times \mathcal T ((P_n,\leqq_n)) \right) = \mathcal T \left ((P_1, \leqq_1) \times \cdots \times (P_n,\leqq_n) \right )$, from which we get the above 
since $\mathcal T \circ \mathcal P = Id_{\mathcal Alex}$.
In other words, the product $\times$ and $\mathcal T$ (similarly $\mathcal P$) commute with each other. i.e., the following diagram commutes:
$$\xymatrix
{
Obj(\mathcal Proset) \times \cdots \times Obj(\mathcal Proset) \ar [d]_{\times } \ar@<0.8ex> [rr]^{\mathcal T \times \cdots \times \mathcal T } && Obj(\mathcal Alex) \times \cdots \times Obj(\mathcal Alex) \ar@<0.8ex> [ll]^{\mathcal P \times \cdots \times \mathcal P }  \ar [d]^{\times} \\
Obj(\mathcal Proset)   \ar@<0.8ex> [rr]^{\mathcal T} && Obj(\mathcal Alex) \ar@<0.8ex> [ll]^{\mathcal P }. \\
}
$$
In fact, the above diagram commutes for the category product:
$$\xymatrix
{
\mathcal Proset \times \cdots \times \mathcal Proset \ar [d]_{\times } \ar@<0.8ex> [rr]^{\mathcal T \times \cdots \times \mathcal T } && \mathcal Alex \times \cdots \times \mathcal Alex \ar@<0.8ex> [ll]^{\mathcal P \times \cdots \times \mathcal P }  \ar [d]^{\times} \\
\mathcal Proset   \ar@<0.8ex> [rr]^{\mathcal T} && \mathcal Alex \ar@<0.8ex> [ll]^{\mathcal P }. \\
}
$$
Here, given categories $\mathcal C_i (1\leqq i \leqq n)$, the category product $\mathcal C_1 \times \cdots \times \mathcal C_n$ is defined by
\begin{itemize}
\item $Obj(\mathcal C_1 \times \cdots \times \mathcal C_n) := \{(X_1, \cdots, X_2) \, | \, X_i \in Obj(\mathcal C_i) \},$
\item  $hom_{\mathcal C_1 \times \cdots \times \mathcal C_n}((X_1, \cdots, X_n) , (Y_1, \cdots, Y_n)):= \{(f_1, \cdots, f_n) \, | \, f_i \in hom_{\mathcal C_i}(X_i, Y_i) \}.$
\end{itemize}
\end{rem}

\section{Decompositions, Decomposition spaces and Prosets}

Let $\mathcal D =\{D_\lambda | \lambda \in \Lambda \}$ be a decomposition of a topological space $X$, i.e.,
\begin{enumerate}
\item $D_\lambda \cap D_\mu = \emptyset$ if $\lambda \not = \mu$,
\item $X = \bigcup_{\lambda \in \Lambda} D_{\lambda}$.
\end{enumerate}
Let $\pi: X \to X/\mathcal D$ be the quotient map.  Let $\tau_{\pi}$ be the quotient topology on the target $X/\mathcal D$. Then the topological space $(X/\mathcal D, \tau_{\pi})$ is called \emph{the decomposition space} and the continuous map $\pi: X \to (X/\mathcal D, \tau_{\pi})$ is called \emph{the decomposition map}. If the content is clear, we sometimes delete the topology $\tau_{\pi}$. The decomposition map $\pi$ is also sometimes denoted simply by $\pi$ for the sake of simplicity.

As to decompositions, R. L. Moore introduced the following notions (\cite{Moore}, cf. \cite{Moore2}):
\begin{defn} Let $\mathcal D$ be a decomposition of a topological space $X$. 
\begin{enumerate}
\item $\mathcal D$ is called \emph{an upper semicontinuous decomposition} if $\bigcup \{D_{\lambda} \, | \, D_{\lambda} \subset U\}$ is open for any open set $U$ of $X$.
\item $\mathcal D$ is called \emph{a lower semicontinuous decomposition} if $\bigcup \{D_{\lambda} \, | \, D_{\lambda} \subset F\}$ is closed for any closed set $F$ of $X$.
\item $\mathcal D$ is called \emph{a continuous decomposition} if it is both upper semicontinuous and lower semicontinuous.
\end{enumerate} 
\end{defn}
\begin{rem}
The decomposition theory of decomposing a (metric) space into continuum (i.e., compact connected space) was developed by R. L. Moore in 1920s and later by R.H. Bing in 1950s (e.g., see \cite{Dav}).  Moore's famous theorem \cite{Moore} is that \emph{if $\mathcal D$ is an upper semicontinuous decomposition of the 2-dimensional Euclidean space $\mathbb R^2$ into continua, none of which separates $\mathbb R^2$, then the decomposition space $\mathbb R^2$ is homeomorphic to the Euclidean space  $\mathbb R^2$: $\mathbb R^2/\mathcal D \cong \mathbb R^2.$} (Here we remark that $\mathcal D$ is an upper (lower) semicontinuous decomposition if and only if the decomposition map $\pi: X \to X/\mathcal D$ is a closed (open) map.)
However, as to the 3-dimensional Euclidean space $\mathbb R^3$, it is not the case, which was proved by R.H. Bing \cite{Bing}: \emph{There exists an upper semicontinuous decomposition $\mathcal D$ of the 3-dimensional Euclidean space $\mathbb R^3$ into continua, none of which separates $\mathbb R^3$, such that the decomposition space $\mathbb R^2/\mathcal D$ is neither homeomorphic to the Euclidean space  $\mathbb R^3$ nor a manifold.} But this decomposition space $\mathbb R^2/\mathcal D$ satisfies that 
$(\mathbb R^3/\mathcal D) \times \mathbb R^1 \cong \mathbb R^4$ (see \cite{Bing2})\footnote{We note that this kind of topology is called \emph{wild topology} and a similar wild topology and decomposition theory was used in M. Freedman's famous proof of the 4-dimensional Poincar\'e conjecture.}. 
This decomposition space $\mathbb R^2/\mathcal D$ is the famous \emph{Bing's dogbone space}.  
\end{rem}

Here we note that it is known that these notions of a decomposition $\mathcal D$ of $X$ can be paraphrased as the properties of the decomposition map $\pi:X \to X/\mathcal D$ as follows:
\begin{pro} Let $\mathcal D$ be a decomposition of a topological space $X$ and let $\pi:X \to X/\mathcal D$ be the decomposition map.
\begin{enumerate}
\item The decomposition $\mathcal D$ is upper semicontinuous if and only if the decomposition map $\pi:X \to X/\mathcal D$ is a closed map.
\item The decomposition $\mathcal D$ is lower  semicontinuous if and only if the decomposition map $\pi:X \to X/\mathcal D$ is an open map.
\item The decomposition $\mathcal D$ is continuous if and only if the decomposition map $\pi:X \to X/\mathcal D$ is an open and closed map.
\end{enumerate}
\end{pro}
\begin{proof} For the sake of simplicity, we give a proof to (2), since we deal with the case when the decomposition map is open in this paper.

That the decomposition map $\pi:X \to X/\mathcal D$ is an open map means that for any open set $G$ of $X$ the image $\pi(G)$ is open in $X/\mathcal D$, which implies by 
the definition of the quotient topology on the decomposition space $X/\mathcal D$ that $\pi^{-1}(\pi(G))$ is open.  Here  we note that
$$\pi^{-1}(\pi(G)) = \bigcup \{D_{\lambda} \, | \, D_{\lambda} \cap G \not = \emptyset \}.$$
Therefore 
\emph{the decomposition map $\pi:X \to X/\mathcal D$ is open if and only if $\bigcup \{D_{\lambda} \, | \, D_{\lambda} \cap G \not = \emptyset \}$ is open for any open set $G$}.
Here we remark that given a subset $G$ of $X$ each $D_{\lambda}$ either intersects $G$ or does not intersect $G$, i.e., either $D_{\lambda}\cap G \not = \emptyset$ or $D_{\lambda} \cap G = \emptyset$, namely $D_{\lambda} \subset G^c =X \setminus G$. Thus we can split the decomposition $\mathcal D$ into two disjoint parts:
$$
X = \bigcup_{\lambda\in \Lambda}  D_{\lambda} = \left (\bigcup \{D_{\lambda} \, | \, D_{\lambda}\cap G \not = \emptyset \} \right) \cup \left (\bigcup \{D_{\lambda} \, | \, D_{\lambda}\subset X \setminus G \} \right).
$$
$\bigcup \{D_{\lambda} \, | \, D_{\lambda}\cap G \not = \emptyset \}$ is open if and only if $\bigcup \{D_{\lambda} \, | \, D_{\lambda}\subset X \setminus G \}$ is closed. If $G$ is open, then $G^c=X\setminus G$ is closed. Therefore  \emph{the decomposition map $\pi:X \to X/\mathcal D$ is open if and only if $\bigcup \{D_{\lambda} \, | \, D_{\lambda} \subset F\}$ is closed for any closed set $F$, i.e., $\mathcal D$ is lower semicontinuous.}

\end{proof}
The following statement is well-known:
\begin{pro}\label{quotient} Let $f: X \to Y$ be a surjective continuous map. If $f$ is open or closed, then the topology of $Y$ is equal to the quotient topology induced by the map $f$.
\end{pro}
\begin{rem} The converse statement does not hold, as seen below.
\end{rem}

Using Propostion \ref{quotient} we can show the following:
\begin{pro}\label{product of decamp} Let $\mathcal D_i \,(1 \leqq i \leqq n)$ be a lower semicontinuous decomposition of a topological space $X_i \, ((1 \leqq i \leqq n))$. Then the product $\mathcal D_1 \times \cdots \mathcal D_n$ is a lower semicontinuous decomposition of the product $X_1 \times \cdots \times X_n$ and we have the homeomorphism
$$(X_1 \times \cdots \times X_n)/(\mathcal D_1 \times \cdots \times \mathcal D_n) \cong (X_1/\mathcal D_1) \times \cdots \times (X_n/\mathcal D_n).$$
\end{pro}
\begin{proof} Since each $\mathcal D_i$ is a lower semicontinuous decomposition of the topological space $X_i$, each decomposition map $\pi_i: X_i \to X_i/\mathcal D_i$ is a continuous and open map. Therefore we have that the product 
$$\pi_1 \times \cdots \times \pi_n: X_1 \times \cdots \times X_n \to (X_1/\mathcal D_1) \times \cdots \times (X_n/\mathcal D_n)$$
is also a continuous and open map\footnote{Here we note that the product of closed maps is not necessarily a closed map, although the product of open maps is an open map. We remark that Proposition \ref{product of decamp} is true if lower semicontinuity is replaced by upper secmicontinuity (e.g., see \cite{Dav}), although in the above proof we just cannot replace ``open map" by ``closed map". }. Now we let $\mathcal D_i =\{D_{i\lambda_i} \, |\, \lambda_i \in \Lambda_i\}$ and we identify $X_i/\mathcal D_i =\Lambda_i$. Then each fiber 
$(\pi_1 \times \cdots \times \pi_n)^{-1}(\lambda_1, \cdots, \lambda_n) = D_{1\lambda_1} \times \cdots D_{n\lambda_n}$. Thus we get
$$\mathcal D_1 \times \cdots \times \mathcal D_n = \{(\pi_1 \times \cdots \times \pi_n)^{-1}(\lambda_1, \cdots, \lambda_n) \, | \, \lambda_1, \cdots, \lambda_n) \in (X_1/\mathcal D_1) \times \cdots \times (X_n/\mathcal D_n)\}.$$ 
Thus it follows from Proposition \ref{quotient} that we have
$$\pi_1 \times \cdots \times \pi_n: X_1 \times \cdots \times X_n \to (X_1/\mathcal D_1) \times \cdots \times (X_n/\mathcal D_n)$$
and the quotient map $\pi:X_1 \times \cdots \times X_n  \to (X_1 \times \cdots \times X_n)/(\mathcal D_1 \times \cdots \times \mathcal D_n)$ is a continuous and open map. Thus the decomposition $\mathcal D_1 \times \cdots \times \mathcal D_n$ is lower semicontinuous.
\end{proof}


\begin{defn} A decomposition $\mathcal D$ of a topological space $X$ such that the decomposition space $X/\mathcal D$ becomes an Alexandroff space is called \emph{an Alexandroff decomposition}. 
\end{defn}

\begin{cor}\label{cor1} Let $\mathcal D_i \,(1 \leqq i \leqq n)$ be a lower semicontinuous Alexandroff decomposition of a topological space $X_i \, ((1 \leqq i \leqq n))$. Then the product $\mathcal D_1 \times \cdots \mathcal D_n$ is a lower semicontinuous Alexandroff decomposition of the product $X_1 \times \cdots \times X_n$ and we have the homeomorphism
$$(X_1 \times \cdots \times X_n)/(\mathcal D_1 \times \cdots \times \mathcal D_n) \cong 
\mathcal T \Bigl((X_1/\mathcal D_1, \leqq_{\tau_{\pi_1}}) \times \cdots \times (X_n/\mathcal D_n, \leqq_{\tau_{\pi_n}}) \Bigr).$$
\end{cor}
\begin{proof} First we observe that since each $\mathcal D_i$ is an Alexandroff decomposition of the topological space $X_i$, we have that $\mathcal T \circ \mathcal P ((X_i/\mathcal D_i, \tau_{\pi_i})) = (X_i/\mathcal D_i, \tau_{\pi_i})$, i.e., $\mathcal T (X_i/\mathcal D_i, \leqq_{\tau_{\pi_i}})) = (X_i/\mathcal D_i, \tau_{\pi_i})$. It follows from the formula (\ref{form.1}) that we have
\begin{align*}
\mathcal T \Bigl ((X_1/\mathcal D_1, \leqq_{\tau_{\pi_1}}) & \times \cdots \times (X_n/\mathcal D_n, \leqq_{\tau_{\pi_n}}) \Bigr ) \\ 
&= \mathcal T (X_1/\mathcal D_1, \leqq_{\tau_{\pi_1}}) \times \cdots \times \mathcal T (X_n/\mathcal D_n, \leqq_{\tau_{\pi_n}})\\
& = (X_1/\mathcal D_1, \tau_{\pi_1}) \times \cdots \times (X_n/\mathcal D_n, \tau_{\pi_n})\\
& = \Bigl ((X_1 \times \cdots \times X_n)/(\mathcal D_1 \times \cdots \times \mathcal D_n), \tau_{\pi_1 \times \cdots \pi_n} \Bigr)
\end{align*}
\end{proof}
\begin{rem} It follows from the above Corollary \ref{cor1} that we can determine the topology of the decomposition space $(X_1 \times \cdots \times X_n)/(\mathcal D_1 \times \cdots \times \mathcal D_n)$ by looking at the proset structure of the product $(X_1/\mathcal D_1, \leqq_{\tau_{\pi_1}}) \times \cdots \times (X_n/\mathcal D_n, \leqq_{\tau_{\pi_n}})$, where the preorder is $\leqq_{\tau_{\pi_1}} \times \cdots \times \leqq_{\tau_{\pi_n}}$, i.e.,  
$$(a_1, \cdots, a_n) \leqq_{\tau_{\pi_1}} \times \cdots \times \leqq_{\tau_{\pi_n}} (b_1, \cdots, b_n) \Longleftrightarrow a_i \leqq_{\tau_{\pi_i}} b_i \, (\forall i).$$
\end{rem} 
Here we give some examples of decomposition spaces and their associated prosets.

\begin{ex}\label{ex1} $\mathcal D = \{ (-\infty, 0), \{0\}, (0, \infty) \}$ is a decomposition of the real line $\mathbb R$. For the quotient map $\pi: \mathbb R\to \mathbb R/\mathcal D$, let $N=\pi((-\infty, 0)), O=\pi(\{0\}), P=\pi((0, \infty)).$
Then the quotient topology for $\mathbb R/\mathcal D$ is
$\tau_{\pi} = \Bigl \{ \emptyset, \{N \}, \{P\}, \{N, P\}, \{N,O,P\} \Bigr \}$ and the proset (in fact poset) $\mathcal P((\mathbb R/\mathcal D, \tau_{\pi}))$ is (we do not write the reflexivity):
$O \leqq N, O \leqq P , \, 
   \xymatrix{
N & O \ar[l] \ar[r] & P
}.
$
\end{ex}
\begin{ex}\label{ex2} $\mathcal D' = \{ (-\infty, -1), [-1,1], (1, \infty) \}$ is another decomposition of $\mathbb R$. For the quotient map $\pi: \mathbb R\to \mathbb R/\mathcal D'$, let $N=\pi((-\infty, -1)), O=\pi([-1,1]), P=\pi((1, \infty)).$
Then the quotient topology for $\mathbb R/\mathcal D'$ is the same as above:
$\tau_{\pi} = \Bigl \{ \emptyset, \{N \}, \{P\}, \{N, P\}, \{N,O,P\} \Bigr \}$.
\end{ex}

\begin{ex}\label{rational} $\mathcal D = \{ \mathbb Q, \mathbb R \setminus \mathbb Q\}$ is a decomposition of the real line $\mathbb R$ into the rational part $\mathbb Q$ and the irrational part $\mathbb R \setminus \mathbb Q$ and for the quotient map $\pi: \mathbb R \to \mathbb R/\mathcal D$ we let 
$q=\pi(\mathbb Q), p=\pi(\mathbb R \setminus \mathbb Q).$
Then the quotient topology for $R/\mathcal D$ is the indiscrete topology:
$\tau_{\pi} = \Bigl \{ \emptyset, \{p, q\} \Bigr \}$ and 
the proset $\mathcal P((\mathbb R/\mathcal D, \tau_{\pi}))$ is: 
$p \leqq q, q \leqq p , \, 
\xymatrix{
    p\ar@/^/[rr] && q \ar@/^/[ll]
}.
$
\end{ex}

\begin{ex}\label{pseudo} For the circle $S^1 = \{(x, y) \in \mathbb R^2 \, | \, x^2+y^2=1\}$, consider the decomposition 
$\mathcal D = \Bigl \{\{(-1, 0)\}, \{(1,0)\},  H^+ =\{(x, y) \in S^1\, | \, y>0\}, H_-=\{(x, y) \in S^1\, | \, y<0\} \Bigr \}$
and the quotient map
$\pi: S^1 \to S^1/\mathcal D$. Let 
$a=\pi((-1, 0)), b=\pi((1, 0)), c= \pi(H^+), d= \pi(H_-).$ Then the quotient topology for $S^1/\mathcal D$ is
$\tau_{\pi} = \Bigl \{ \emptyset, \{c \}, \{d\}, \{c, d\}, \{a, c, d\}, \{b, c, d\}, \{a, b, c, d\} \Bigr \}$ and 
the proset (in fact poset) $\mathcal P((S^1/\mathcal D, \tau_{\pi}))$ is 
$$a \leqq c, \quad b \leqq c, \quad a \leqq d, \quad b \leqq d,
   
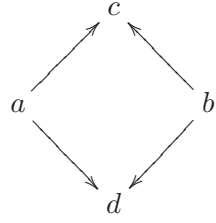
\xymatrix{
& c \\
a \ar[ur] \ar[dr] && b \ar[ul] \ar[dl]\\
& d}
$$
This four-point poset is well-known as \emph{the pseudo-circle}, denoted $\mathbb S^1$, which is weakly homotopic to the standard circle $S^1$, i.e., $\pi_n(\mathbb S^1) \cong \pi_n(S^1)$ for any $n \geqq 1$.
In fact their homology and cohomology groups are also isomorphic, since if $f:X \to Y$ is a weakly homotopy equivalence, then $f_*:H_*(X) \cong H_*(Y)$ and $f^*:H^*(X) \cong H^*(Y)$ (e.g., see \cite[Proposition 4.21]{Ha}).
\end{ex}

\begin{ex}\label{pseudo'} For $S^1 = \{(x, y) \in \mathbb R^2 \, | \, x^2+y^2=1\}$, consider another decomposition 
$\mathcal D' = \Bigl \{\{(-1, 0)\}, B= \left \{(x, y) \in S^1\, \Bigl | \, \frac{2}{3} \leqq x \leqq 1 \right \}, C= H^+ \setminus B, D= H_- \setminus B  \Bigr \}$ 
and the quotient map
$\pi: S^1 \to S^1/\mathcal D$. Here $H^+, H_-$ are as the above example. Let 
$a=\pi((-1, 0)), b=\pi(B), c= \pi(C), d= \pi(D)$. Then
the quotient topology for $S^1/\mathcal D'$ is the same as above:
$\tau_{\pi} = \Bigl \{ \emptyset, \{c \}, \{d\}, \{c, d\}, \{a, c, d\}, \{b, c, d\}, \{a, b, c, d\} \Bigr \}.$
\end{ex}

\begin{ex} \label{ex6}
$\mathcal D:= \left\{ \{(x,0) | x <0\}, \{(0,0)\}, \{(x, 0) | x >0\}, \{(x,y) | y>0\}, \{(x,y) | y<0\} \right \} $ is a decomposition of $\mathbb R^2$.
For the quotient map
$\pi: \mathbb R^2 \to \mathbb R^2/\mathcal D$, let 
$a=\pi(\{(x,0) | x <0\}), o =\pi((0,0)), b=\pi(\{(x, 0) | x >0\}), c= \pi(\{(x,y) | y>0\}), d= \pi(\{(x,y) | y<0\}).$
Then the quotient topology for $\mathbb R^2/\mathcal D$ is
$\tau_{\pi} = \{\emptyset, \{c\}, \{d\}, \{c, d\}, \{a, c, d\}, \{b, c, d\}, \{a, b, o, c, d\}\}$ and 
the proset (in fact poset) $\mathcal P((\mathbb R^2/\mathcal D, \tau_{\pi}))$ is: 
$o \leqq a, o \leqq b,  a \leqq c, o \leqq c,  b \leqq c,  a \leqq d,  o \leqq d,  b \leqq d.$
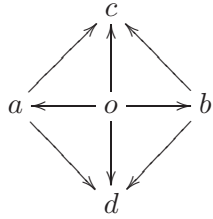
\[
   \xymatrix{
& c\\
a \ar[ur] \ar[dr] & o \ar[u] \ar[d] \ar[l] \ar[r] & b \ar[ul] \ar[dl]\\
& d
}
\]
\end{ex}

\begin{ex}\label{ex7} For $\mathbb R^2$ consider another following decomposition:
\begin{align*}
\mathcal D':= & \Bigl \{\{(x,y) | x <0, y>0\}, \{(0,y) | y>0\}, \{(x,y) | x >0, y>0\} \\
& \, \, \, \, \{(x,0) | x <0\},  \qquad \, \, \, \, \, \{(0,0)\},  \qquad \, \, \, \, \{(x,0) | x >0\}, \\
& \, \, \, \, \{(x,y) | x <0, y<0\}, \{(0,y) | y<0\}, \{(x,y) | x >0, y<0\} \Bigr \} 
\end{align*}
For the quotient map $\pi: \mathbb R^2 \to \mathbb R^2/\mathcal D$, we let 
\begin{align*}
& p_{-+} = \pi(\{(x,y) | x <0, y>0\}), p_{0+} = \pi(\{(0,y) | y>0\}), p_{++} = \pi(\{(x,y) | x >0, y>0\}) \\
&  p_{-0} \,\,  = \pi(\{(x,0) | x <0\}),  \qquad \, \, \, \, p_{00} \, \, = \pi(\{(0,0)\}),  \qquad \, \, \, \, p_{+0} \, \, = \pi(\{(x,0) | x >0\}), \\
& \, p_{--} = \pi(\{(x,y) | x <0, y<0\}), p_{0-} = \pi(\{(0,y) | y<0\}), p_{+-} = \pi(\{(x,y) | x >0, y<0\}).
\end{align*}
Then the quotient topology for $\mathbb R^2/\mathcal D$ consists of a lot of open sets (in fact, 50 open sets:writing down them is left for the reader) and its base is:
\begin{align*}
& \Bigl \{\mathbb R^2/\mathcal D, \{ p_{-+}\}, \{ p_{++} \},\{p_{+-}\}, \{p_{--}\},\\
& \qquad \qquad \{ p_{-+}, p_{0+}, p_{++} \}, \{ p_{--}, p_{0-}, p_{+-} \},\{ p_{-+}, p_{-0}, p_{--} \}, \{ p_{++}, p_{+0}, p_{+-} \}  \Bigr\}
\end{align*}
The proset (in fact poset) $\mathcal P((\mathbb R^2/\mathcal D, \tau_{\pi}))$ is:
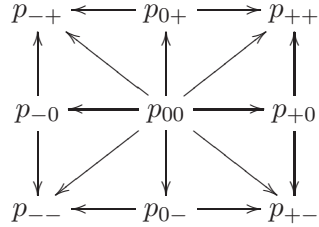
\begin{equation}\label{3x3}
   \xymatrix{
p_{-+}& p_{0+} \ar[l] \ar[r] & p_{++}\\
p_{-0} \ar[u] \ar[d] & p_{00} \ar[ul] \ar[ur] \ar[dl] \ar[dr] \ar[u] \ar[d] \ar[l] \ar[r] & p_{+0} \ar[u] \ar[d]\\
p_{--}& p_{0-} \ar[l] \ar[r] & p_{+-}
}
\end{equation}
\end{ex}

In fact, the above Example \ref{ex7} is a very special case of the following:
\begin{ex}\label{coordinate hyperplane}(see also Example \ref{arrangement} below) Take the product of $n$-copies of the decomposition $\mathcal D = \{(-\infty, 0), \{0\}, (0, \infty)\}$ of $\mathbb R$ in Example \ref{ex1}; $\mathcal D^n = \overbrace{ \mathcal D \times \cdots \times \mathcal D}^{n}$ is a decomposition of $\mathbb R^n$. The case when $n=2$ is nothing but the above Example \ref{ex7}.
It follows from Corollary \ref{cor1} that we have
\begin{align*}
\mathbb R^n/\mathcal D^n & = \mathcal T \Bigl ((\mathbb R/\mathcal D, \leqq_{\tau_{\pi}}) \overbrace{ \times \cdots \times}^{n} (\mathbb R/\mathcal D, \leqq_{\tau_{\pi}}) \Bigr)\\
& = \mathcal T \Bigl ( (\{N, O, P\}, \leqq) \overbrace{ \times \cdots \times}^{n} (\{N, O, P\}, \leqq) \Bigr)\\
& = \mathcal T \Bigl ( (\{N, O, P\} \overbrace{ \times \cdots \times}^{n} \{N, O, P\}, \leqq \overbrace{ \times \cdots \times}^{n}  \leqq ) \Bigr)
\end{align*}
In the case when $n=2$, the proset (in fact poset) $(\{N, O, P\} \times \{N, O, P\}, \leqq \times \leqq )$ is the following, which is the same as the poset (\ref{3x3}) with just the different symbols of the points used:
\begin{equation}
   \xymatrix{
(N,P) & (O,P) \ar[l] \ar[r] & (P,P) \\
(N,O)  \ar[u] \ar[d] & (O,O) \ar[ul] \ar[ur] \ar[dl] \ar[dr] \ar[u] \ar[d] \ar[l] \ar[r] & (P,O) \ar[u] \ar[d]\\
(N,N) & (O,N) \ar[l] \ar[r] & (P,N).
}
\end{equation}
From this poset structure we can determine all the open sets of $\mathbb R^2/\mathcal D^2$. 

In the case when $n$ is 3, we just write down the following poset, which is a part of the whole poset $(\{N, O, P\} \times \{N, O, P\} \times \{N, O, P\}, \leqq \times \leqq \times \leqq )$: 
{\tiny  \begin{equation}
   \xymatrix{
& & (P,P,P) & (O,P,P)\ar[l] \ar[r] &  (N,P,P)\\
& (P,O,P)\ar[dl] \ar[ur]& (O,O,P) \ar[l] \ar[u] \ar[d] \ar[dl] \ar[dll] \ar[l] \ar[ur] \ar[r] \ar[urr] &  (N,O,P) \ar[dl] \ar[ur]& (N,P,O) \ar[u] \ar[d]\\
(P, N,P) & (O,N,P) \ar[l] \ar[r] & (N,N, P) & (N, O,O) \ar[u] \ar[l] \ar[uur] \ar[ur] \ar[r]\ar[d] \ar[r] \ar[dl] \ar[ddl]& (N,P,N)\\
(P, N,O)  \ar[u] \ar[d] & (O,N,O) \ar[ul] \ar[ur] \ar[dl] \ar[dr] \ar[u] \ar[d] \ar[l] \ar[r] & (N,N,O) \ar[u] \ar[d] & (N,O,N) \ar[ur] \ar[dl]\\
(P, N,N) & (O,N, N) \ar[l] \ar[r] & (N,N,N).
}
\end{equation}
}
\end{ex}

So far the decompositions are all finite, i.e., the number of pieces of the decomposition is finite, thus the decomposition spaces are all finite topological spaces. Here we give an example in which the decomposition space is infinite.
\begin{ex}\label{infinity} Let $\mathbb R^{\infty} =\{(x_1, x_2, \cdots, x_i, \cdots) \, | \, x_i \in \mathbb R\}$ be the infinite dimensional Euclidean space. We consider the decomposition $\mathcal D = \{D_n \}$ such that
\begin{enumerate}
\item $D_0 = \text{the origin},$
\item For $n\geqq 1$, $D_n := \{(x_i)_{i\geqq 1} \, | \, x_n \not = 0, x_j =0 \, (\forall j >n) \}.$
\end{enumerate}
For the quotient map $\pi: \mathbb R^{\infty} \to \mathbb R^{\infty}/\mathcal D$ let $n := \pi(D_n)$. Then the decomposition space is the natural numbers $\mathbb N$ with the quotient topology, which is the Alexandroff topology associated with the total order $0 \leqq 1 \leqq 2 \leqq 3 \leqq \cdots \leq n \leqq \cdots$. 
\end{ex}

When it comes to the question of whether the above decomposition maps are open maps or not, the answers are the following:
They are all open maps, except for Example \ref{ex2} and Example \ref{pseudo'}. 
It is not that easy to check it directly. For example, it is easy to see that the quotient map $\pi_{\mathcal D'}: \mathbb R \to \mathbb R/\mathcal D'$ in Example \ref{ex2} is \emph{not} an open map, since the image $\pi_{\mathcal D'}((0,1)) = \{O\}$ is a closed set, not an open set, although $(0,1)$ is an open set of $\mathbb R$. However the quotient map $\pi_{\mathcal D}: \mathbb R \to \mathbb R/\mathcal D$ in Example \ref{ex1} \emph{is} an open map. One needs to prove it; for example the proof goes as follows. Let $U$ be an open set of $\mathbb R$. Then either $0 \in U$ or $0 \not \in U$. If $0 \not \in U$, then $\pi_{\mathcal D}(U)$ is $\{N\}$, $\{P\}$ or $\{N, P\}$,  which is an open set. If $0 \in U$, then by the definition of an open set of $\mathbb R$, there exists an open interval $(-\epsilon, \epsilon)$ such that $0 \in (-\epsilon, \epsilon) \subset U$ where $\epsilon >0$, hence $U \cap (-\infty, 0) \not = \emptyset$ and $U \cap (0, \infty) \not = \emptyset$. Therefore, if $U$ is an open set and $0 \in U$, then $\pi_{\mathcal D}(U) = \{N, O, P\}$, which is an open set. Thus the quotient map $\pi_{\mathcal D}: \mathbb R \to \mathbb R/\mathcal D$ is an open map. We can imagine that if a given decomposition $\mathcal D$ of a topological space $X$ has lots of pieces, say 100 pieces, then it would be not easy or quite tedious to check whether the quotient map $\pi_{\mathcal D}: X \to X/\mathcal D$ is open or not. In fact we can easily determine the openness of the quotient map via the proset-structure of the quotient space $X/\mathcal D$, which we discuss below.

In \cite{Tamaki} Dai Tamaki proves the following \emph{``preorder versus frontier-condition" criterion for being open} :
\begin{thm}[D.Tamaki]\label{tamaki} Let $\Lambda$ be a poset and let $\pi: X \to \Lambda$ be a surjective continuous map for the Alexandroff topology on $\Lambda$. Let $D_{\lambda}:=\pi^{-1}(\lambda)$. $\pi$ is open if and only if $\la \leqq \mu \Longleftrightarrow D_{\la} \subset \overline{D_{\mu}}$.
\end{thm}
\begin{rem} In his theorem the target $\Lambda$ is a poset, however it can be a proset.
\end{rem}

\begin{cor}\label{tamaki-cor} Let $\mathcal D= \{D_{\lambda}\}_{\lambda \in \Lambda}$ be a decomposition of a topological space $X$ such that the decomposition space $X/\mathcal D$  becomes an Alexandroff space and let $\leqq_{\tau_{\pi}}$ be the preorder of the proset $\mathcal P((X/\mathcal D, \tau_{\pi}))$ associated to the Alexandroff space. Then the decomposition map
$\pi: X \to X/\mathcal D =\Lambda$ is open if and only if $\la \leqq_{\tau_{\pi}} \mu \Longleftrightarrow D_{\la} \subset \overline{D_{\mu}}$.
\end{cor}
Thus, in the above examples, it is very easy to see if the decomposition map $\pi:X \to X/\mathcal D$ is open or not, simply by checking if $\la \leqq_{\tau_{\pi}} \mu \Longleftrightarrow D_{\la} \subset \overline{D_{\mu}}$ holds or not.\\

We can define another preorder on the quotient set $X/\mathcal D =\Lambda$ by
$$\la \leqq^* \mu \Longleftrightarrow D_{\la} \subset \overline {D_{\mu}}.$$
Then it follows from Proposition \ref{quotient} and Theorem \ref{tamaki} above that the decomposition map $\pi: X \to X/\mathcal D$ is open if and only if the proset $\mathcal P((X/\mathcal D, \tau_{\pi}))$ is the same as the proset $(X/\mathcal D, \leqq^*)$. 

Now, when $(X/\mathcal D, \leqq^*)$ is defined as above, we have the Alexandroff space $\mathcal T((X/\mathcal D, \leqq^*)) = (X/\mathcal D, \tau_{\leqq^*})$ and a natural question is
$$ \text{\emph{Is the quotient map $\pi: X \to (X/\mathcal D, \tau_{\leqq^*})$ continuous?}}$$

The following is an answer to this question:
\begin{thm}\label{answer}Let $\mathcal D = \{D_{\lambda} \}_{\lambda \in \Lambda}$ be a decomposition of a topological space $X$. The quotient map $\pi: X \to \mathcal T((X/\mathcal D, \leqq^*))=(X/\mathcal D, \tau_{\leqq^*})$ is continuous if and only if $\mathcal T((X/\mathcal D, \leqq^*)) = (X/\mathcal D, \tau_{\leqq^*})$ is the decomposition space $(X/\mathcal D, \tau_{\pi})$ (thus $\mathcal D$ is an Alexandroff decomposition).
\end{thm}

\begin{proof} Suppose that the quotient map $\pi: X \to \mathcal T((X/\mathcal D, \leqq^*))$ is continuous. By the definition we have  $\la \leqq^* \mu \Longleftrightarrow D_{\la} \subset \overline {D_{\mu}}.$ Then it follows from Theorem \ref{tamaki} that $\pi: X \to \mathcal T((X/\mathcal D, \leqq^*))$ is open. Therefore it follows Proposition \ref{quotient} that the topology of $\mathcal T((X/\mathcal D, \leqq^*))$ is equal to the quotient topology of the quotient map $\pi:X \to X/\mathcal D$, i.e., $\mathcal T((X/\mathcal D, \leqq^*))$ is the decomposition space. Thus $\mathcal D$ is an Alexandroff decomposition. The ``if" part is clear.
\end{proof}

As an application or an example of Theorem \ref{answer},  we discuss real hyperplane arrangements. 
\begin{ex}\label{arrangement} Let $\mathcal A =\{H_1, H_2, \cdots, H_k \}$ be a real hyperplane arrangement of $\mathbb R^n$. Here $H_i$ is a hyperplane defined by an affine form or a linear polynomial $\ell_i=a_{i0}+a_{i1}x_1+ \cdots +a_{ir}x_r+ \cdots + a_{in}x_n$: $H_i= \{(x_1, x_2, \cdots, x_n) \in \mathbb R^n \, | \, \ell_i (x_1, x_2,\cdots, x_n) = 0 \}$. The hyperplane arrangement defines the decomposition $\mathcal D(\mathcal A)$, which is obtained as follows: We let
$$H_i^- := \{ (x_1, \cdots, x_n) \, | \, \ell_i(x_1, \cdots, x_n) <0\}, H_i^+ := \{ (x_1, \cdots, x_n) \, | \, \ell_i(x_1, \cdots, x_n) >0\}.$$
We also let $ H_i^0:=H_i.$ Then
$$\mathcal D (\mathcal A):= \left \{\cap_{i=1}^k A_i \, | \, A_i \in \{H_i^-, H_i^0, H_i^+\} \right \}.$$
The above two examples Example \ref{ex1} and Example \ref{ex7} are such decompositions obtained by hyperplane arrangements in $\mathbb R$ and $\mathbb R^2$, respectively.
Here we note that some $\bigcap_{i=1}^k A_i$ can be empty, which is then deleted.
The poset $\mathcal P((\mathbb R^n/\mathcal D(\mathcal A), \tau_{\pi}))$ of the decomposition space $(\mathbb R^n/\mathcal D(\mathcal A), \tau_{\pi})$ is nothing but the so-called \emph{face poset} $F(\mathcal A)$, which is the oriented matroid (see \cite{OT})\footnote{In \cite{OT} the partial order $\leqq$ is the reversed one. To get the same situation as in \cite{OT} we just define the Alexandroff topology via ``down-set" instead of ``up-set".}. This can be seen as follows.
Let us consider the following continuous map:
$$\Phi: \mathbb R^n \xrightarrow {\ell_1 \times \ell_2 \times \cdots \times \ell_k} \mathbb R^k \xrightarrow {\pi \times \pi \times \cdots \times \pi} \{N, O, P\}^k,$$
where $\pi:\mathbb R \to \{N,O, P\}$ is the decomposition map in Example \ref{ex1}. 
For each piece $\cap_{i=1}^k A_i$ of the above decomposition $\mathcal D (\mathcal A)$ we have
$$\Phi \left (\cap_{i=1}^k A_i \right) = \Bigl ( \pi(\ell_1(A_1)), \cdots, \pi(\ell_i(A_i)), \cdots, \pi(\ell_k(A_k) \Bigr ),$$
where
$$
\pi(\ell_i(A_i)) = \begin{cases}
P & \, A_i = H_i^+, \\
O & \, A_i = H_i^0 =H_i, \\
N & \, A_i = H_i^-.
\end{cases}
$$
Then the quotient map $\pi: \mathbb R^n \to \mathbb R^n/\mathcal D(\mathcal A)$ is considered as the map 
$$\Phi: \mathbb R^n \to \op{Im}\Phi \subset \{N, O, P\}^k,$$
 since each piece $\cap_{i=1}^k A_i$ of the decomposition $\mathcal D (\mathcal A)$ is mapped to a point. Since $\Phi$ is the composite of continuous maps $\ell_1 \times \cdots \times \ell_k$ and $\pi \times \cdots \times \pi$, $\Phi: \mathbb R^n \to \op{Im}\Phi $ is a continuous map. Here we emphasize that the topological space $\op{Im}\Phi$ is a subspace of the Alexandroff space $\mathcal T((\{N, 0, P\}, \leqq)^k )) = (\{N, 0, P\}, \tau_{\leqq})^k $ associated to the product $(\{N, 0, P\}, \leqq)^k = (\{N, 0, P\}\times \cdots \times \{N, 0, P\}, \leqq \times \cdots \times \leqq)$ of the poset $(\{N, 0, P\}, \leqq)$. Here $(a_1, \cdots a_k) (\leqq \times \cdots \times \leqq)  (b_1, \cdots, b_k) \Leftrightarrow a_i \leqq b_i (\forall i)$ (see Proposition \ref{product}), where $a_i, b_i \in \{N, 0, P\}$. 
Here we observe that for two pieces $\cap_{i=1}^k A_i$ and $\cap_{i=1}^k A'_i$, where $A_i, A'_i \in \{H_i^-, H_i^0, H_i^+\}$,
\begin{align}
\Phi\left(\cap_{i=1}^k A_i \right) (\leqq \times \cdots \times \leqq) \Phi\left(\cap_{i=1}^k A'_i \right) & \Longleftrightarrow \pi(\ell_i(A_i)) \leqq \pi(\ell_i(A'_i)) (\forall i)\\
& \Longleftrightarrow \cap_{i=1}^k A_i \subset \overline{\cap_{i=1}^k A'_i}. \label{ob}
\end{align}
Therefore it follows from Theorem \ref{answer} that the topology of the topological space $\op{Im}\Phi$ is nothing but the quotient topology of the map  $\Phi: \mathbb R^n \to \op{Im}\Phi$, i.e., the quotient topology of the quotient map $\pi: \mathbb R^n \to \mathbb R^n/\mathcal D(\mathcal A)$, therefore the topological space $\op{Im}\Phi$ is the decomposition space $(\mathbb R^n/\mathcal D(\mathcal A), \tau_{\pi})$ and the associated proset (in fact poset) $\mathcal P((\mathbb R^n/\mathcal D(\mathcal A), \tau_{\pi}))$ is nothing but the face poset $F(\mathcal A)$ or equivalently the associated Alexandroff space $\mathcal T(F(\mathcal A))$ is nothing but the decomposition space $(\mathbb R^n/\mathcal D(\mathcal A), \tau_{\pi})$.
\end{ex}

\begin{rem} In fact, Example \ref{coordinate hyperplane} above is a special case of the above hyperplane arrangement, namely it is the case of the so-called \emph{coordinate hyperplane arrangement}: $\mathcal A = \{ \{x_1 =0\}, \{x_2 =0\}, \cdots, \{x_n =0\} \}$. In this special case, the affine or continuous map $\ell_1 \times \cdots \ell_n: \mathbb R^n \to \mathbb R^n$ defining the continuous map $\Phi: \mathbb R^n \xrightarrow {\ell_1 \times \ell_2 \times \cdots \times \ell_n} \mathbb R^n\xrightarrow {\pi \times \pi \times \cdots \times \pi} \{N, O, P\}^n$ is nothing but the identity, thus we get that $\Phi$ is equal to $\pi \times \pi \times \cdots \times \pi: \mathbb R^n \to \{N, O, P\}^n$, which is treated in Example \ref{coordinate hyperplane}.
\end{rem}

\section{stratifications and poset-stratified spaces}

A stratification of a topological space (which can be the underlying topological one of a much finer object such as a complex algebraic variety, a complex analytic space) is a special kind of decomposition with certain extra conditions. 
There seems to be no fixed or standard definition of \emph{stratification} and there are several ones depending on the objects to study, such as topologically stratified spaces and Thom--Whitney stratified spaces. In \cite{Tamaki} D. Tamaki gives a nice review of several stratifications available in mathematics.

Here is one definition of stratification:
\begin{defn}\label{stratification}
Let $X$ be a topological space. If a family $\{D_{\la}\}_{\la \in \Lambda}$ of subsets satisfies the following conditions, then $\{D_{\la}\}_{\la \in \Lambda}$  is called \emph{a stratification} of $X$.
\begin{enumerate}
\item $D_{\la} \cap D_{\mu} = \emptyset$ if $\la \not = \mu.$
\item $X = \bigcup_{\la} D_{\la}$.
\item (locally closed set) Each $D_{\la}$ is a locally closed set.
\item (frontier condition) $D_{\la} \cap \overline{D_{\mu}} \not = \emptyset  \Longrightarrow D_{\la} \subset \overline{D_{\mu}}.$
($\bar {A} - A$ is called the frontier, thus if $\la \not = \mu$, i.e. $D_{\la} \cap D_{\mu} = \emptyset$, then this condition implies that $D_{\la} \cap \overline{D_{\mu}} \not = \emptyset  \Longrightarrow D_{\la} \subset \overline{D_{\mu}}- D_{\mu}$, i.e., $D_{\la}$ is contained in the frontier of $D_{\mu}$.)
\end{enumerate}
\end{defn}

In fact, the frontier condition is ``basically" a sufficient condition for the continuity of $\pi: X \to \mathcal T((X/\mathcal D, \leqq^*))=(X/\mathcal D, \tau_{\leqq^*})$. The following proposition was observed by Hiro Lee Tanaka \cite{Tam2}:
\begin{pro} \label{Tanaka} Let $X$ be a topological space and let $\pi: X \to \Lambda$ be a surjective map to a set $\Lambda$, and let $D_{\lambda} := \pi^{-1}(\lambda)$ and we define the preorder by $\lambda \leqq \mu \Leftrightarrow D_{\lambda} \subset \overline {D_{\mu}}.$
If the following two conditions hold, then the map $\pi: X \to \Lambda$ is continuous for the Alexandroff topology for $\Lambda$:
\begin{enumerate}
\item (frontier condition) if $D_{\lambda} \cap \overline {D_{\mu}} \not = \emptyset$, then $D_{\lambda} \subset \overline {D_{\mu}}$.
\item For any closed subset $C \subset \Lambda$, $\bigcup_{\lambda \in C} \overline {D_\lambda}$ is closed. (Note that if $\Lambda$ is a finite set, then this condition is automatic.)
\end{enumerate}
\end{pro}

So far, we have not discussed a poset-structure of the proset $\mathcal P((X/\mathcal D, \tau_{\pi}))$ of the decomposition space $(X/\mathcal D, \tau_{\pi})$. As Example \ref{rational} shows, the proset $\mathcal P((X/\mathcal D, \tau_{\pi}))$ is not necessarily a poset. A necessary condition is the following:

\begin{lem} Let $\mathcal D =\{D_{\lambda} | \lambda \in \Lambda\}$ be an Alexandroff decomposition of a topological space $X$. If the proset $\mathcal P((X/\mathcal D, \tau_{\pi}))$ of the decomposition space $(X/\mathcal D, \tau_{\pi})$ is a poset, then each piece $D_{\lambda}$ is locally closed. 
\end{lem}
\begin{proof} It follows from Corollary \ref{locally closed} that for a poset $\Lambda$ each singleton $\{\lambda\}$ is locally closed in the Alexandroff topology $\mathcal T(\Lambda)$, say $\lambda = F \cap G$ where $F$ is closed and $G$ open. Hence for the continuous map $\pi: X \to X/\mathcal D$, $D_{\lambda} =\pi^{-1}(\lambda)$ is locally closed because $D_{\lambda} = \pi^{-1}(\lambda) = \pi^{-1}(F \cap G)=\pi^{-1}(F) \cap \pi^{-1}(G)$, where 
$\pi^{-1}(F)$ is closed and $\pi^{-1}(G)$ is open. 
\end{proof}

\begin{thm}\label{open-locally} Let $\mathcal D =\{D_{\lambda} | \lambda \in \Lambda\}$ be a decomposition of a topological space $X$ such that the decomposition map $\pi: X \to X/\mathcal D$ is open. Then, if each piece $D_{\lambda}$ is locally closed, then the proset $\mathcal P((X/\mathcal D, \tau_{\pi}))$ is a poset.
\end{thm}
\begin{cor} Let Let $\mathcal D =\{D_{\lambda} | \lambda \in \Lambda\}$ be an Alexandroff decomposition of a topological space $X$ and suppose that the decomposition map $\pi: X \to X/\mathcal D$ is open. Then the proset $\mathcal P((X/\mathcal D, \tau_{\pi}))$ is a poset if and only each piece $D_{\lambda}$ is locally closed.
\end{cor}
\begin{rem} Indeed, in the case of Example \ref{rational}, one can show that the rational part $\mathbb Q$ and the irrational part $\mathbb R \setminus \mathbb Q$ are both not locally closed. For example, suppose that $\mathbb Q$ is locally closed, i.e., $\mathbb Q= U \cap F$ where $U$ is an open set and $F$ a closed set. Then $\mathbb Q \subset F$ and since $\mathbb Q$ is dense, $\mathbb R = \overline {\mathbb Q} \subset \overline{F} = F$, thus $F = \mathbb R$. Hence $\mathbb Q = U \cap F = U \cap \mathbb R = U$, thus $\mathbb Q$ is an open set, contradicting the fact that $\mathbb Q$ is not open. Therefore $\mathbb Q$ is not locally closed. Similarly, the irrational part $\mathbb R \setminus \mathbb Q$ is also not locally closed.
\end{rem}
The above Theorem \ref{open-locally} follows from Corollary \ref{tamaki-cor} above and the following proposition (cf. \cite{Hu}):
\begin{pro}\label{poset} Let $\mathcal D =\{D_{\lambda} | \lambda \in \Lambda\}$ be an Alexandroff decomposition of a topological space $X$ such that each piece $D_{\lambda}$ is locally closed. If we define the preorder $\la \leqq^* \mu \Longleftrightarrow D_{\la} \subset \overline {D_{\mu}},$
then this preorder is a partial order, i.e. $(X/\mathcal D, \leqq^*)$ is a poset.
\end{pro}
\begin{proof} We show that $\lambda \leqq^* \mu$ and $\mu \leqq^* \lambda$ impy $\lambda = \mu$, i.e., $
D_{\lambda} \subset \overline{D_{\mu}}$ and $D_{\mu} \subset \overline{D_{\lambda}}$ imply $D_{\lambda} = D_{\mu}.$ Since $D_{\la}$ is locally closed, $D_{\la} = U \cap F$ where $U$ is an open set and $F$ is a closed set. Hence, $D_{\la}= U \cap D_{\la}$ is a closed set in $U$. Hence the closure of $U \cap D_{\la}$ in $U$ is
$\overline{U \cap D_{\la}}^U = U \cap D_{\la}$. Since $\overline{U \cap D_{\la}}^U = U \cap \overline{D_{\la}}$ where $\overline{D_{\la}}$ is the closure of $D_{\la}$ in $X$. Hence we have
$U \cap D_{\la} = U \cap \overline{D_{\la}}$. It follows from $D_{\mu} \subset \overline{D_{\la}}$ that
$U \cap D_{\mu} \subset U \cap \overline{D_{\la}} = U \cap D_{\la}.$
Choose a point $x \in D_{\la}$. Then $x \in D_{\la} \subset \overline{D_{\mu}}$. Since $U$ is an open set containing the point $x$ and $x \in \overline{D_{\mu}}$, it follows that $\emptyset \not = U \cap D_{\mu} \subset U \cap \overline{D_{\la}} = U \cap D_{\la}$. Which implies that $D_{\la} \cap D_{\mu} \not = \emptyset$, therefore we get $D_{\la} = D_{\mu}.$
\end{proof}

\begin{cor} Let $\mathcal D = \{D_{\la} | \la \in \Lambda\}$ be a finite stratification (i.e., $|\Lambda| <\infty$) as defined above:
\begin{enumerate}
\item $D_{\la} \cap D_{\mu} = \emptyset$ if $\la \not = \mu.$
\item $X = \bigcup_{\la} D_{\la}$.
\item (locally closed) Each $D_{\la}$ is a locally closed set
\item (frontier condition) $D_{\la} \cap \overline{D_{\mu}} \not = \emptyset  \Longrightarrow D_{\la} \subset \overline{D_{\mu}}.$
\end{enumerate}
Then the decomposition map $\pi:X \to X/\mathcal D$ is a continuous map to a poset with the Alexandroff topology.
\end{cor}
\begin{proof} It follows from Proposition \ref{poset} that the proset $(X/\mathcal D, \leqq^*)$ is a poset and furthermore it follows from Proposition \ref{Tanaka} that $\pi: X \to \mathcal T((X/\mathcal D, \leqq^*))$ is continuous. Furthermore it follows from Theorem \ref{answer} that $\mathcal T((X/\mathcal D, \leqq^*))$ is the decomposition space $(X/\mathcal D, \tau_{\pi})$. Hence we get the above result.
\end{proof}

Such a continuous map from a topological space to a poset considered as a topological space with the Alexandroff topology seems to be an interesting object to study, as treated in recent papers (e.g., \cite{AFT}, \cite{Curry}, \cite{Lurie}, \cite{Tamaki}, etc.)

In Example \ref{ex2}, the decomposition is not a stratification in the above sense and the decomposition map is not an open map, but it is a continuous map to a poset with the Alexandroff topology.

\begin{defn}[poset-stratified space] Let $P$ be a poset. A \emph{poset-stratified space $X$} over the poset $P$ is a pair $(X, X \xrightarrow s P)$ of a topological space $X$ and a continuous map $s: X \to P$ where $P$ is considered as the associated Alexandroff space. 
\end{defn}

\begin{rem} The notion of poset-stratified space seems to be due to Jacob Lurie \cite{Lurie}. For a poset-stratified space $(X, X \xrightarrow s P)$, $X$ is the underlying topological space and $s: X \to P$ is considered  as \emph{a structure of poset-stratification}. If the context is clear, then we just write a poset-stratified space $X$, just like writing a topological space $X$ without referring to which topology to be considered on it. 
\end{rem}

\begin{ex}[$CW$-complex] Let $X$ be a $CW$-complex with $X_{\leqq k}$ denoting the $k$-skeleton, i.e., $X_{\leqq k} \setminus X_{\leqq k-1}$ is the disjoint union of open cells of dimension $k$ (each cell of dimension $k$ is homeomorphic to the $k$-dimensional Euclidean space $\mathbb R^k$). Then $s: X \to \mathbb Z_{\geqq 0}$ defined by $s(X_{\leqq k} \setminus X_{\leqq k-1}) :=k$ is continuous. Here $\mathbb Z_{\geqq 0}$ is the set of non-negative integers with the usual total roder. Thus any $CW$-complex is a poset-stratified space and its structure of poset-stratification is nothing but counting dimensions of open cells.
\end{ex}

The category of poset-stratified spaces is denoted by $\mathcal Strat$. The objects are pairs $(X, X \xrightarrow {s} P)$ of a topological space $X$ and a continuous map $s: X \to P$ from the space $X$ to a poset $P$ with the Alexandroff topology associated to the poset $P$. Given two poset-stratified spaces $(X, X \xrightarrow {s} P)$ and $(X', X' \xrightarrow {s'} P)$, a morphism from $(X, X \xrightarrow {s} P)$ to $(X', X' \xrightarrow {s'} P')$ is a pair of a continuous map $f:X \to X'$  and a monotone map $q:P \to P'$ (i.e., for $a\leqq b$ in $P$ we have $q(a) \leqq q(b)$ in $P'$, thus it is a continuous map for the associated Alexandroff spaces) such that the following diagram commutes:
$$
\xymatrix{
X \ar [r]^{s} \ar [d]_{f} & P \ar [d]^{q} \\
X' \ar[r]_{s'} & P'. \\
}
$$

\section{a poset-stratified-space structure of the set $hom_{\mathcal C}(X,Y)$ of morphisms}

In this section we show that for any locally small category $\mathcal C$ the set $hom_{\mathcal C}(X,Y)$ can be considered as a poset-stratified space in a natural way.

First we observe the following:
\begin{lem} \label{mono} Given a proset $(P, \leqq)$, we define the following relation on $P$:
$$a, b \in P, a \sim b \Longleftrightarrow a \leqq b \, \, \text {and}\, \, b \leqq a.$$
\begin{enumerate}
\item 
The relation $\sim$ is an equivalence relation and we denote the set of the equivalence classes by $P/_{\sim}$. 
\item Then we define the order $\leqq'$ on $P/_{\sim}$ as follows: for $[a], [b] \in P/_{\sim}$
$$[a] \leqq' [b] \Longleftrightarrow a \leqq b.$$
Then this is well-defined, i.e., it does not depend on the representatives $a$ and $b$.
\item 
The proset $(P_{\sim}, \leqq')$ is a poset, i.e., $[a] \leqq' [b]$ and $[b] \leqq' [a]$ imply that $[a] = [b]$.
\item \label{mono2}The projection or quotient map
$\pi: (P, \leqq) \to (P_{\sim}, \leqq')$ defined by $\pi(a):=[a]$
is a monotone map. 
\end{enumerate}
\end{lem}
\begin{proof} 
We give a proof for the sake of completeness, although it is straightforward.
\begin{enumerate}
\item It is due to the definition of $a \sim b $.
\item Indeed, if $a' \sim a$ and $b' \sim b$, thus $a' \leqq a, a \leqq a'$ and $b' \leqq b, b \leqq b'$, then $a' \leqq a \leqq b \leqq b'$, which implies that $a' \leqq b'$.
\item  By the definitions they imply that $a \leqq b$ and $b \leqq a$, thus $a \sim b$, which implies that $[a] = [b]$.
\item  Indeed, since $a \leqq b$ implies that $[a] \leqq' [b]$, i.e., $\pi(a) \leqq' \pi(b)$.
\end{enumerate}
\end{proof}
\begin{thm}\label{theorem} Let $(P, \leqq)$ and $(P_{\sim}, \leqq')$ be as above.
\begin{enumerate}
\item For the Alexandroff topologies the quotient map $\pi: (P, \tau_{\leqq}) \to (P_{\sim}, \tau_{\leqq'})$ 
is an open map. Hence, the Alexandroff topology of the poset $(P_{\sim}, \leqq')$ is the same as the quotient topology of the above quotient map $\pi: (P, \leqq) \to P_{\sim}.$ 
\item In particular, each equivalence class $\{b \in P \, \, | \, a \sim b\}$ of $a$, i.e., the fiber $\pi^{-1}([a])$,  is a locally closed set in the Alexandroff topology of the proset $(P, \leqq)$.
\item In particular, $[a] \leqq' [b]$ if and only if $[a] \subset \overline {[b]}$, where we consider $[a], [b]$ as subsets in $P$.
\end{enumerate}
\end{thm}
\begin{proof} (1) Suppose that $U$ is an open set in $(P, \tau_{\leqq})$. We want to show that $\pi(U)$ is an open set in $(P_{\sim}, \leqq')$, i.e., $[a] \in \pi(U), [a] \leqq' [b] \Longrightarrow [b] \in \pi(U)$. $[a] \in \pi(U)$ implies that $\exists a' \in U$ such that $[a] = \pi(a') =[a']$, i.e. $a' \leqq a$ (and $a \leqq a'$). Since $[a] \leqq' [b]$ we have $a \leqq b$. Since $U$ is an open set in $(P, \tau_{\leqq})$, thus $a' \in U$ and  $a' \leqq a \leqq b$ imply that $b \in U$, therefore we get that $[b] = \pi(b) \in \pi(U).$ Thus $\pi(U)$ is open. It follows from the above Lemma \ref{mono} (\ref{mono2}) that the quotient map $\pi: (P, \leqq) \to (P_{\sim}, \leqq')$ is a monotone map. Hence it follows from Lemma \ref{monotone} that for the Alexandroff topologies the map $\pi: (P, \tau_{\leqq}) \to (P_{\sim}, \tau_{\leqq'})$ is a continuous map. Therefore it follows from Proposition \ref{quotient} that the Alexandroff topology of the poset $(P_{\sim}, \leqq')$ is the same as the quotient topology of the above quotient map $\pi: (P, \leqq) \to P_{\sim}.$

(2) Since $(P_{\sim}, \leqq')$ is a poset, it follows from Corollary \ref{locally closed} that each singleton $\{[a]\}$ is a locally closed set. Hence the inverse image $\pi^{-1}([a])$ of the locally closed set $\{[a]\}$ is also a locally closed set, thus each equivalence class $[a]=\{b \in P \, \, | \, a \sim b\}$ of $a$ is locally closed as a subset.

(3) Since $\pi: (P, \tau_{\leqq}) \to (P_{\sim}, \tau_{\leqq'})$ is a continuous and open map, the statement follows from Theorem \ref{tamaki}.
\end{proof}


\begin{lem} Let $\mathcal C$ be a locally small category.
\begin{enumerate}
\item  On the set $hom_{\mathcal C}(X,Y)$ we define 
$g \leqq_R f$ by $\exists s \in hom_{\mathcal C}(X, X)$ such that $f = g \circ s.$
i.e. the diagram $\xymatrix
{
X \ar [d]_s \ar [r]^{f} & Y \\
X \ar [ur] _{g} & }
$
 commutes. This order is a preorder.
\item On the set $hom_{\mathcal C}(X,Y)$ we define the relation 
$f \sim_R g$ by $g \leqq_R f$ and $f \leqq_R g$, which mean that 
$\exists s_1, s_2 \in hom_{\mathcal C}(X, X)$ such that $f = g \circ s_1$ and $g = f \circ s_2$, i.e., the following diagram commutes:
$$\xymatrix
{
X \ar@<.5ex> [d]^{s_1} \ar [r]^{f} & Y.\\
X \ar@<.5ex> [u]^{s_2} \ar [ur] _{g} }
$$
$\sim_R$ is an equivalence relation. The equivalence class of $f$ is denoted by $[f]_R$.
\item The partial order on the quotient $hom_{\mathcal C}(X,Y)_R:= hom_{\mathcal C}(X,Y)/_{\sim_R}$ is well-defined as 
$$[g]_R \leqq'_R [f]_R \, \Longleftrightarrow \, \exists s \in hom_{\mathcal C}(X,X) \, \text{such that} \,\,\, f = g \circ s.$$
Thus $hom_{\mathcal C}(X,Y)_R:= hom_{\mathcal C}(X,Y)/_{\sim_R}$ is a poset with the above order.
\item $\pi_R:(hom_{\mathcal C}(X,Y), \leqq_R) \to (hom_{\mathcal C}(X,Y)_R,\leqq'_R)$ defined by $\pi_R(f):= [f]_R$ is a monotone map.
\end{enumerate}
\end{lem}

Thus from the above Lemma \ref{mono} and Theorem \ref{theorem} we get the following theorem:
\begin{thm} Let $\mathcal C$ be a locally small category and let the set-up be as above.
\begin{enumerate}
\item For any objects $X,Y \in \mathcal Obj(\mathcal C)$ the canonical quotient map
$$\pi_R:(hom_{\mathcal C}(X,Y), \tau_{\leqq_R}) \to (hom_{\mathcal C}(X,Y)_R,\leqq'_R)$$
is a poset-stratified space for the Alexandroff topologies.
\item In other words, $\mathcal D:= \{[f]_R\}$ is a decomposition of $hom_{\mathcal C}(X,Y)$ such that $[f]_R$ (as a subset) is a locally closed set in the Alexandroff space $(hom_{\mathcal C}(X,Y), \tau_{\leqq_R})$.
\item $[g]_R \leqq'_R [f]_R$ if and only if $[g]_R \subset \overline {[f]_R}$ as subsets in $(hom_{\mathcal C}(X,Y), \tau_{\leqq_R})$.
\end{enumerate}
\end{thm}

\begin{cor} Let $\mathcal C$ be a locally small category. For any object $S \in Obj(\mathcal C)$, we have an associated covariant functor
$\frak {st}^S_*: \mathcal C \to \mathcal Strat$
such that
\begin{enumerate}
\item for each object $Y \in Obj(\mathcal C)$, 
$$\frak {st}^S_*(X) := \left ( (hom_{\mathcal C}(S,X), \tau_{\leqq_R}), (hom_{\mathcal C}(S,X), \tau_{\leqq_R}) \xrightarrow {\pi_R} (hom_{\mathcal C}(S,X)_R, \leqq'_R) \right)$$
\item for a morphism $f: X \to Y$ , $\frak {st}^S_*(f)$ is the following commutative diagram:
$$\xymatrix
{
(hom_{\mathcal C}(S,X), \tau_{\leqq_R}) \ar [r]^{\pi_R} \ar [d]_{f_*} & (hom_{\mathcal C}(S,X)_R, \leqq'_R)  \ar [d]^{f_*} \\
(hom_{\mathcal C}(S,Y), \tau_{\leqq_R})  \ar[r]_{\pi_R} & (hom_{\mathcal C}(S,Y)_R, \leqq'_R) \\
}
$$
\end{enumerate}
\end{cor}
Similarly we can define the following:
\begin{lem} Let $\mathcal C$ be a locally small category.
\begin{enumerate}
\item   On the set $hom_{\mathcal C}(X,Y)$ we define the following order $g \leqq_L f$ by $\exists t \in hom_{\mathcal C}(Y,Y)$ such that $f = t \circ g.$
i.e. the diagram $\xymatrix
{
X \ar [dr]_g \ar [r]^{f} & Y \\
& Y \ar [u] _{t}}
$ commutes. This order is a preorder.
\item On the set $hom_{\mathcal C}(X,Y)$ we define the relation $f \sim_L g$ by $g \leqq_L f$ and $f \leqq_L g$, which mean that 
$\exists t_1, t_2 \in hom_{\mathcal C}(Y,Y)$ such that $f = t_1 \circ g$ and $g = t_2 \circ f$, i.e., the following diagram commutes:
$$\xymatrix
{
X \ar [r]^{f} \ar[dr]_g & Y \ar@<.5ex> [d]^{t_2}\\
& Y. \ar@<.5ex> [u]^{t_1}}
$$
$\sim_L$ is an equivalence relation. The equivalence class of $f$ is denoted by $[f]_L$.
\item The partial order on the quotient $hom_{\mathcal C}(X,Y)_L:= hom_{\mathcal C}(X,Y)/_{\sim_L}$ is well-defined as 
$$[g]_L \leqq'_L [f]_L \quad \Longleftrightarrow \quad \exists t \in hom_{\mathcal C}(Y,Y) \,\,\, \text{such that} \,\,\, f = t \circ g.$$
Thus $hom_{\mathcal C}(X,Y)_L:= hom_{\mathcal C}(X,Y)/_{\sim_L}$ is a poset with the above order.
\item $\pi_L:(hom_{\mathcal C}(X,Y), \leqq_L) \to (hom_{\mathcal C}(X,Y)_L,\leqq'_L)$ defined by $\pi_L(f):= [f]_L$ is a monotone map.
\end{enumerate}
\end{lem}

\begin{thm} Let the set-up be as above.
\begin{enumerate}
\item For any objects $X,Y \in \mathcal Obj(\mathcal C)$ the canonical quotient map
$$\pi_L:(hom_{\mathcal C}(X,Y), \tau_{\leqq_L}) \to (hom_{\mathcal C}(X,Y)_L,\leqq'_L)$$
is a poset-stratified space for the Alexandroff topologies.
\item In other words, $\mathcal D:= \{[f]_L\}$ is a decomposition of $hom_{\mathcal C}(X,Y)$ such that $[f]_L$ (as a subset) is a locally closed set in the Alexandroff space $(hom_{\mathcal C}(X,Y), \tau_{\leqq_L})$.
\item $[g]_L \leqq'_L [f]_L$ if and only if $[g]_L \subset \overline {[f]_L}$ as subsets in $(hom_{\mathcal C}(X,Y), \tau_{\leqq_L})$.
\end{enumerate}
\end{thm}
\begin{cor} Let $\mathcal C$ be a locally small category. For any object $T \in Obj(\mathcal C)$, we have an associated contravariant functor
$\frak {st}_T^*: \mathcal C \to \mathcal Strat$
such that
\begin{enumerate}
\item for each object $X \in Obj(\mathcal C)$, 
$$\frak {st}_T^*(X) := \left ( (hom_{\mathcal C}(X,T), \tau_{\leqq_L}), (hom_{\mathcal C}(X,T), \tau_{\leqq_L}) \xrightarrow {\pi_L} (hom_{\mathcal C}(X,T)_L, \leqq'_L) \right)$$
\item for a morphism $f: X \to Y$ , $\frak {st}_T^*(f)$ is the following commutative diagram:
$$\xymatrix
{
(hom_{\mathcal C}(Y,T), \tau_{\leqq_L}) \ar [r]^{\pi_L} \ar [d]_{f^*} & (hom_{\mathcal C}(Y,T)_L, \leqq'_L)  \ar [d]^{f^*} \\
(hom_{\mathcal C}(X,T), \tau_{\leqq_L})  \ar[r]_{\pi_L} & (hom_{\mathcal C}(X,T)_L, \leqq'_L) \\
}
$$
\end{enumerate}
\end{cor}

If we mix the above two, we get the following:
\begin{lem}  Let $\mathcal C$ be a locally small category.
\begin{enumerate}
\item  On the set $hom_{\mathcal C}(X,Y)$ we define the order
$g \leqq_{LR} f$ by $\exists s \in hom_{\mathcal C}(X,X)$ and $\exists t \in hom_{\mathcal C}(Y,Y)$ such that $f = t \circ g \circ s.$
i.e. the diagram $\xymatrix
{
X \ar [d]_s \ar [r]^{f} & Y \\
X \ar [r] _{g} & Y \ar [u]_t}
$ commutes. 
This order is a preorder.
\item On the set $hom_{\mathcal C}(X,Y)$ we define the relation 
$f \sim_{LR} g$ by $g \leqq_{LR} f$ and $f \leqq_{LR} g$, which mean that $\exists s_1, s_2 \in hom_{\mathcal C}(X,X)$ and
$\exists t_1, t_2 \in hom_{\mathcal C}(Y,Y)$ such that $f = t_1 \circ g \circ s_1$ and $g = t_2 \circ f \circ s_2$, i.e., the following diagram commutes:
$$\xymatrix
{
X \ar [r]^{f} \ar@<.5ex> [d]^{s_1} & Y \ar@<.5ex> [d]^{t_2}\\
X \ar[r]_g \ar@<.5ex> [u]^{s_2} & Y. \ar@<.5ex> [u]^{t_1}}
$$
$\sim_{LR}$ is an equivalence relation. The equivalence class of $f$ is denoted by $[f]_{LR}$.
\item The partial order on the quotient $hom_{\mathcal C}(X,Y)_{LR}:= hom_{\mathcal C}(X,Y)/_{\sim_{LR}}$ is well-defined as 
$$[g]_{LR} \leqq'_{LR} [f]_{LR} \Longleftrightarrow  \exists s \in hom_{\mathcal C}(X,X) , \, \,  \exists t \in hom_{\mathcal C}(Y,Y) \,\,\, \text{such that} \,\,\, f = t \circ g \circ s.$$
Thus $hom_{\mathcal C}(X,Y)_{LR}:= hom_{\mathcal C}(X,Y)/_{\sim_{LR}}$ is a poset with the above order.
\item $\pi_{LR}:(hom_{\mathcal C}(X,Y), \leqq_{LR}) \to (hom_{\mathcal C}(X,Y)_{LR},\leqq'_{LR})$ defined by $\pi_{LR}(f):= [f]_{LR}$ is a monotone map.
\end{enumerate}
\end{lem}

\begin{thm} Let the set-up be as above.
\begin{enumerate}
\item For any objects $X,Y \in \mathcal Obj(\mathcal C)$ the canonical quotient map
$$\pi_{LR}:(hom_{\mathcal C}(X,Y), \tau_{\leqq_{LR}}) \to (hom_{\mathcal C}(X,Y)_{LR},\leqq'_{LR})$$
is a poset-stratified space for the Alexandroff topologies.
\item In other words, $\mathcal D:= \{[f]_{LR}\}$ is a decomposition of $hom_{\mathcal C}(X,Y)$ such that $[f]_{LR}$ (as a subset) is a locally closed set in the Alexandroff space $(hom_{\mathcal C}(X,Y), \tau_{\leqq_{LR}})$.
\item $[g]_{LR} \leqq'_{LR} [f]_{LR}$ if and only if $[g]_{LR} \subset \overline {[f]_{LR}}$ as subsets in $(hom_{\mathcal C}(X,Y), \tau_{\leqq_{LR}})$.
\end{enumerate}
\end{thm}

\begin{rem} For this mixed situation we cannot get any functor from $\mathcal C$ to $\mathcal Strat$, unlike the cases of $hom_{\mathcal C}(X,Y)_R, hom_{\mathcal C}(X,Y)_L$.
\end{rem}

\begin{rem} This construction gives a kind of universal poset-stratified space structure $\pi: (hom_{\mathcal C}(X,A), \leqq_L) \to (hom_{\mathcal C}(X,A)_{\sim}, \leqq_L')$ to the hom-set $hom_{\mathcal C}(X,A)$, and the monotone map $ \mathcal Im_{F}: (hom_{\mathcal C}(X,A), \leqq_L) \to (\mathscr P(F^*(X)), \leqq))$ involving another contravariant functor $F:\mathcal C^{op} \to \mathcal Set$ gives \emph{a more geometric one}, so to speak.
\end{rem}

\begin{rem} In \cite{YY2} we consider the above in the case of the homotopy category $h\mathcal Top$ of topological spaces, thus we have
$hom_{h\mathcal Top}(X,Y)_R =[X,Y]_R$, $hom_{h\mathcal Top}(X,Y)_L=[X,Y]_L$ and $hom_{h\mathcal Top}(X,Y)_{LR}=[X,Y]_{LR}.$
For continuous maps $f, g: X \to Y$, for example, the above relation $[g]\leqq_L [f]$ defined by $\exists t: Y \to Y$ such that $[f] = [t] \circ [g]$, i.e., $f \sim t \circ g$ seems to be an abstract nonsense or artificial. However,
Jim Stasheff (private communication) informed us that this kind of thing, in a different context, was already considered by Karol Borsuk \cite{Bor1, Bor2} and Peter Hilton \cite{Hil1} (cf. \cite{Hil2, Hil3}) in 1950's. According to these papers, we can sum up as follows:
\begin{enumerate}
\item K. Borsuk introduced \emph {dependence of maps}: $f: X \to Y$ is said to \emph{depend on $g:X \to Y$} if whenever $g$ is extended to $X' \supset X$, so is $f$. He gave an alternative naming for this notion: \emph{$f$ is a multiple of $g$} or \emph{$g$ is a divisor of $f$}. it turned out that this naming was correct, because Borsuk proved that \emph{$f$ depends on $g$ if and only if there exists a map $h: Y \to Y$ such that $f \cong h \circ g$}.

\item Borsuk defined two maps $f$ and $g$ to be \emph{conjugate} if they depend on each other, i.e., by our notation $[f] \leqq_L [g]$ and $[g] \leqq_L [f]$, i.e., $[f] \sim_L [g]$.

\item $f:X \to Y$ is said to \emph{co-depend on} $g:X \to Y$ if whenever $g$ lifts to the total space $E$ of a fibration over $Y$, so does $g$. Then the dual of the above result is that \emph{ $f$ co-depends on $g$ if and only if there exists a map $h: X \to X$ such that $f \cong  g \circ h$}. 
\item It is natural to define that if $f$ and $g$ co-depends on each other, they are called \emph{co-conjugate}. We are not sure if Borsuk or Hilton defined the notion of co-conjugate.
\item \emph{The above results about the co-dependence marks the birth of Eckmann-Hilton duality!}
\item In fact, \emph{R. Thom \cite{Thom} independently introduced the notion of dependence of cohomology classes}. Thom's dependence is subsumed in Borsuk's dependence.
\end{enumerate}
Thus, using Borsuk's notion, \emph{$[X,Y]_R$} and \emph{$[X, Y]_L$} are the poset of the homotopy classes of \emph{co-conjugate} maps and \emph{conjugate} maps, resp. Furthermore \emph{$[X,Y]_{LR}$} can be considered as the poset of homotopy classes of \emph{conjugate-co-conjugate} maps.
\end{rem}

\section {A remark on Yoneda's Lemma}

In \cite{YY2} we consider the following \emph{more geometric} poset-stratified space structure on the homotopy set $[X,Y]$.

Let $\mathscr P(S)$ be a power set of a set $S$, i.e., the set of all the subsets of the set $S$. By the order $S_1 \leqq S_2 \Longleftrightarrow S_1 \subset S_2$, $(\mathscr P(S), \leqq)$ and $(\mathscr P(G), \leqq)$ are clearly posets, since the order $\leqq$ is a partial order.

By considering the cohomology theory $H^*(-; \mathbb Z)$, we get a canonical monotone map 
$$\op{Im}_{H^*}:([X,Y], \leqq_L) \to (\mathscr P(H^*(X)), \leqq),$$
which is defined by $\op{Im}_{H^*}([f]):= \op{Im}(f^*:H^*(Y) \to H^*(X)) =f^*(H^*(Y))$.
This \emph{monotone} map has a connection with 
\begin{enumerate}
\item R. Thom's notion of \emph{dependence of cohomology classes} \cite{Thom}, if we consider $Y = K(\mathbb Z, p)$ the Eilenberg-Maclane space:
$$\op{Im}_{H^*}:([X,K(\mathbb Z, p)], \leqq_L) \to (\mathscr P(H^*(X)), \leqq),$$ 
\item the ring $\mathbb Z[c_1(E), \cdots, c_n(E)]$ of all the characteristic classes of complex vector bundles of $n= rank (E)$, if we consider $Y= G_n(\mathbb C^{\infty})$ the infinite Grassmann of $n$-dimensional planes in $\mathbb C^{\infty}$:
$$\op{Im}_{H^*}:\op{Vect}_n(X) \cong ([X,G_n(\mathbb C^{\infty})], \leqq_L) \to (\mathscr P(H^*(X)), \leqq).$$
\end{enumerate}
For more details, see \cite{YY2}.

In this section we will show that the above two cases are special ones of an observation on the well-known Yoneda's Lemma.

Yoneda's lemmas about representable functors are the following:
\begin{thm}Let $\mathcal C$ be a locally small category, i.e., $hom_{\mathcal  C}(A,B)$ is a set for any objects $A, B \in Obj(\mathcal  C)$, and let $\mathcal Set$ be the category of sets.

\begin{enumerate}
\item (the covariant case) Let $F:\mathcal  C \to \mathcal Set$ be a covariant functor. Let $h_A:= hom_{\mathcal  C}(A, -)$ be a covariant hom-set functor $h_A: \mathcal  C \to \mathcal Set$. Then the set of all the natural transformations from the hom-set covariant functor $h_A =hom_{\mathcal  C}(A,-)$ to the covariant functor $F$ is isomorphic to the set $F(A)$:
$$\mathcal Natural (h_A, F) \cong F(A).$$

\item (the contravariant case) Let $F:\mathcal  C^{op} \to \mathcal Set$ be a contravariant functor. Let $h^A:= hom_{\mathcal  C}(-, A)$ be a contravariant hom-set functor $h^A: \mathcal  C \to \mathcal Set$. Then the set of all the natural transformations from the hom-set contravariant functor $h^A =hom_{\mathcal  C}(-, A)$ to the contravariant functor $F$ is isomorphic to the set $F(A)$:
$$\mathcal Natural (h^A, F) \cong F(A).$$
\end{enumerate}
\end{thm}

The contravariant case of Yoneda's Lemma is proved by using the following commutative diagram: Let $\tau: hom_{\mathcal  C}(-, A) \to F(-)$ be a natural transformation:
\begin{equation}
\xymatrix
{
\op{id}_A\in hom_{\mathcal  C}(A, A)\ar[rr]^{\tau} \ar[d]_{f^*} && F(A) \ar[d]^{f^*\qquad } \ni \tau (\op{id}_A) \\
f \in hom_{\mathcal  C}(X, A) \ar[rr]_{\tau} && F(X) \ni \tau(f).
}
\end{equation}
Note that $f = f^*(\op{id}_A)= f \circ \op{id}_A$. Hence we have
\begin{align*}
\tau(f) & =  \tau (f^*(\op{id}_A)) \\
& = f^*(\tau (\op{id}_A)) \, \, \text {(by the naturality of $\tau$)}
\end{align*}
Thus the natural transformation $\tau: hom_{\mathcal  C}(-, A) \to F(-)$ is determined by the element $\tau(\op{id}_A) \in F(A)$. Conversely, given any element $\alp \in F(A)$ we can define the natural transformation $\tau_{\alp}:hom_{\mathcal  C}(-, A) \to F(-)$ by, for each object $X \in Obj(\mathcal  C)$
$$\tau_{\alp}: hom_{\mathcal  C}(X, A) \to F(X) \quad \tau_{\alp}(f) := f^*\alp,$$
in which case $\tau_{\alp}:hom_{\mathcal  C}(A, A)  \to F(A)$ satisfies that $\tau_{\alp}(\op{id}_A) = \op{id}_A^*(\alp)=\alp$.
The above isomorphism map is called the Yoneda map:
$$\mathcal Y: \mathcal Natural (h^A, F) \cong F(A) \quad \mathcal Y(\tau):=\tau(\op{id}_A),\text{or}$$
$$\mathcal Y: F(A) \cong  \mathcal Natural (h^A, F) \quad \mathcal Y(\alp):=\tau_{\alp}.$$
Let $h^A(-)=hom_{\mathcal  C}(-, A):\mathcal  C \to \mathcal Set$ and $F:\mathcal  C^{op} \to \mathcal Set$ be as above. Then for each object $X\in Obj(\mathcal  C)$ we have the following canonical map:
$$ \mathcal Im_{F}: h^A(X)=hom_{\mathcal C}(X, A)  \to \mathcal P(F(X))$$ defined by
$$\mathcal Im_{F}(f) := \op{Image} \Bigl (f^*:F(A) \to F(X) \Bigr ) = f^*(F(A)) = \{f^*\alp \, | \, \alp \in F(A) \}.$$
The last two parts are written down for an emphasis. As observed in the above, $f^*\alp = \tau_{\alp}(f)$
which is the image of $f$ under the natural transformation $\tau_{\alp}$ corresponding to $\alp \in F(A)$.
In other word 

$\mathcal Im_{F}(f)$ \emph{is the set consisting of the images of $f$ by \underline {all} the natural transformations $\mathcal Natural (h^A, F)$}. 
For a morphism $g:X \to Y \in \mathcal C$, we have the following commutative diagram:

$$\xymatrix
{
h^A(Y)=hom_{\mathcal C}(Y, A) \ar [rr]^{\qquad \mathcal Im_{F}} \ar [d]_{g^*} && \mathcal P(F(Y)) \ar [d]^{g^*} \\
h^A(X)=hom_{\mathcal C}(X, A)  \ar[rr]_{\qquad \mathcal Im_{F}} && \mathcal P(F(X)) \\
}
$$
If we let $\mathscr PF: \mathcal C^{op} \to \mathcal Set$ be the ``subset" functor associated to the given functor $F:\mathcal C^{op} \to \mathcal Set$, defined by for an object $X$, $\mathscr PF(X) := \mathscr P (F(X))$ and for a morphism $g:X \to Y$, $\mathscr PF(g):\mathscr PF(Y) \to \mathscr PF(X)$  defined by $\mathscr PF(g)(S):=g^*(S)$ for $S \subset F(Y)$, we can consider $\mathcal Im_{F}(f)$ as \emph{a natural transformation}
$$\mathcal Im_{F}: h^A(-) \to \mathscr PF (-),$$
which \emph{sort of ``collects" all the natural transformation images}.

The upshot is the following. 
\begin{ob} \emph{Let $\mathcal C$ be a locally small category and $\mathcal Set$ be the category of sets. Let
$h^A(-)=hom_{\mathcal C}(-, A) : \mathcal C \to \mathcal Set$ be a representable contravariant functor and $F:\mathcal C^{op} \to \mathcal Set$ be another contravariant functor. Then we have the following canonical natural transformation (sort of ``collecting" or ``using" all the natural transformations from $h^A$ to $F$)
$$\mathcal Im_{F}:h^A(-) \to \mathscr PF(-),$$
where  for each object $X \in Obj(\mathcal C)$ we have $\mathcal Im_{F}(f) =f^*(F(A)) (\subset F(X) )$, which is the set consisting of the images of $f$ by \underline{all} the natural transformations from $h^A$ to $F$.}
\end{ob}
The similar observation for the covariant case is made mutatis mutandis, so omitted.

Since $f \leqq_L g$ for $f, g \in hom_{\mathcal C}(X,A)$ implies that $\mathcal Im_{F}(f) \leqq \mathcal Im_{F}(g)$, we get the following:
\begin{cor} Let the situation be as above. For any object $X$ 
$$\mathcal Im_{F}:(hom_{\mathcal C}(X,A), \leqq_L) \to (\mathscr PF(-), \leqq)$$
is a monotone map from a proset to a poset, i.e., a poset-stratified space.
\end{cor}

\begin{rem} Depending on the situations, the target category of our contravariant functor can have more structures, e.g., groups, abelian groups, rings, commutative rings, etc.
\end{rem}


\begin{rem}One can consider some other reasonable or interesting pairs $(h^A(-), F(-))$ of representable contravariant functors $h^A(-)$ and contravariant functors $F(-)$. \\
\end{rem}

\noindent
{\bf Acknowledgements:} The author would like to thank Jim Stasheff, Dai Tamaki and Toshihiro Yamaguchi for useful comments. 
This work was supported by JSPS KAKENHI Grant Numbers 16H03936 and JP19K03468.

\end{document}